\documentclass[12pt]{amsart}
\usepackage{amssymb,amsmath,amsthm,latexsym,booktabs,todonotes,color,comment,eucal,uncial,comment,leftidx}
\usepackage[normalem]{ulem}
\usepackage{amsfonts}
\usepackage{graphicx, graphics,relsize,url}

\usepackage{tikz}

 \newcommand{\mex}{{\bf mex}}
 \newcommand{\Gr}{{\bf Gr}}
\def\N{\mathcal{N}}
\def\P{\mathcal{P}}

\theoremstyle{definition}
\newtheorem{definition}{Definition}

\newtheorem{remark}[definition]{Remark}

\theoremstyle{plain}
\newtheorem{lemma}[definition]{Lemma}

\newtheorem{theorem}[definition]{Theorem}
\newtheorem{corollary}[definition]{Corollary}

\begin{document}

\title[Two-pile and three-pile games of Halve Nim]{Two-pile and three-pile games of \\
a new variant of Nim known as \\
Halve Nim}

\author[S. Locke]{Stephen C. Locke}
\address{Department of Mathematics and Statistics, 
         Florida Atlantic University, 
         777 Glades Rd,  
         Boca Raton, FL 33431 
         USA }
\email{lockes@fau.edu}
\author[S. Curran]{Stephen J. Curran}
\address{Department of Mathematics
         University of Pittsburgh at Johnstown
         450 Schoolhouse Rd
         Johnstown, PA 15904
         USA}
\email{sjcurran@pitt.edu}
\author[R. Low]{Richard M. Low}
\address{Department of Mathematics, 
         San Jose State University, 
         1 Washington Sq, 
         San Jose, CA 95192 
         USA }
\email{richard.low@sjsu.edu}


\date{June 28, 2026. }

\begin{abstract}
We investigate a variant of Nim called Halve Nim, which 
in addition to the standard moves of Nim, 
we allow replacing each pile of coins 
with half its amount.  
We determine the $\mathcal{P}$-positions of all three-pile games of Halve Nim
in which one pile has
at most ten coins.
This result extends known results on
the $\mathcal{P}$-positions of the two-pile games of Halve Nim. 
\end{abstract}

\maketitle      

\section{Introduction and Preliminaries}\label{sec:Intro}

Having its humble beginnings in the context of 
recreational mathematics, combinatorial game theory has 
matured into an active area of research. Along with its 
natural appeal, the subject has applications to 
complexity theory, logic, graph theory and biology. For 
these reasons, combinatorial games have caught the 
attention of many people and the large body of research 
literature on the subject continues to increase. The 
interested reader is directed to 
\cite{AlbNowak,AlNoWo,BCG,Conway,Guy,GN2,L2025,L2019,N2015,N94}, 
and to A.~Fraenkel's excellent bibliography \cite{Fraen}.

A \textit{combinatorial game} is one of complete information and no element of chance is involved in gameplay. Each player is aware of the game position at any point in the game. Under \textit{normal play}, two players alternate taking turns and a player loses when he cannot make a move. An \textit{impartial} combinatorial game is one where both players have the same options from any position. A \textit{finite} game eventually terminates (with a winner and a loser, no draws allowed). 

Perhaps the most famous finite impartial combinatorial game is Nim, which is played in the following manner: 
\begin{itemize}
\item[] There are $n$ piles, each containing a finite number of coins. Two players alternate turns, each time choosing a pile and removing any number ($\geq 1$) of coins in that pile. The player who cannot make a move loses the game.
\end{itemize}

\noindent
A complete analysis of Nim was given by Bouton \cite{Bou}. Since then, a large number of papers have been written on variants of Nim and are found within the mathematical literature.

In this paper, we analyze a variant of Nim, which we call \textbf{Halve Nim}. It is played in the following way  between Player 1 (who moves first) and Player 2: 

\begin{itemize}
\item[$\bullet$] In front of them are various piles of coins. At each turn, a player makes a move. If there are no moves available, that player loses and the opponent wins the game. Each move is one of two types:
\begin{enumerate}
    \item The player may perform a standard Nim move. Here, the player chooses one pile and removes at least one coin from that pile.
    \item The player can {\it halve} the game, replacing every pile of $2k$ or $2k+1$ coins by a pile  of $k$ coins, for every possible $k$.
\end{enumerate}
\end{itemize}

\noindent
Note that the game must eventually end.
Each move   decreases the number of coins.
When there are no coins, the player whose turn it is (has no moves), loses.

First, we recall some concepts from combinatorial game theory. Any terms which are not explicitly defined in this paper can be found in \cite{AlNoWo}. Let $\mathbb{N}$ denote the set of nonnegative integers. For $S\subset {\mathbb{N}}$ , with $S\ne {\mathbb{N}}$, we define
$\mex(S) = \min \{t\in {\mathbb{N}} : t\notin S\}$. If $G$ is a game position in
which there are no moves available, we set $\Gr(G)=0$.
Then, if $G$ is any other position in the game and
$Q$ is the set of
positions that can be reached from $G$ in one move, then
\begin{equation*}
\Gr(G)=\mex\{ \Gr(H) : H\in Q \}
\end{equation*}
and this is called the {\it Grundy number} of $G$.
 A game $G$ that can be won by the player that made the
 move previous to $G$ is called a $\P$-\textit{position}.
 Similarly, we say that $G$ is an $\N$-\textit{position} if
 $G$ can be won by the player to make the next move after $G$.
 Thus, a game $G$ is a $\P$-position
 if and only if $\Gr(G)=0$. 

 \begin{remark} \label{rem:FindPposition}
     The proof technique that we use
     to demonstrate that a game $G$ is
     either a $\P$-position or an
     $\N$-position is based on the following observations.
     \begin{itemize}
         \item[$\bullet$] A game $G$ is a $\P$-position if for each move
         $H$ from $G$, there is a move $K$ from $H$ such that
         $K$ is a $\P$-position.
         We will sometimes equivalently show that the move $K$ is an $\N$-position.
         \item[$\bullet$] A game $G$ is an $\N$-position if there exists a
         move $H$ from $G$ such that $H$ is a $\P$-position.
     \end{itemize}
 \end{remark} 

\indent
We denote a position $G$ in Halve Nim by
$G= 1^{a_1} \, 2^{a_2} \cdots k^{a_k}$, where $G$ has $a_j$ piles of size $j$.
Thus, for example, $1^4 3^5 7^2$ denotes a position with four piles of
one coin, five piles of three coins, and two piles of seven coins.

Bouton \cite{Bou} completely characterized the winning strategy for Nim.
A game of Nim with piles of coins of heights $a_1,a_2,\ldots,a_n$
is a $\mathcal{P}$-position if and only if
$a_1 \oplus a_2 \oplus \cdots \oplus a_n=0$.
One reason to believe that Halve Nim will produce results similar to
those of standard Nim is that the halving operation
respects bitwise addition.
Namely, we have
\begin{equation*}
    \biggl\lfloor\frac{a_1 \oplus a_2 \oplus \cdots \oplus a_n}{2}
    \biggr\rfloor=
    \biggl\lfloor\frac{a_1}{2} \biggr\rfloor  \oplus
    \biggl\lfloor\frac{a_2}{2} \biggr\rfloor  \oplus \cdots\oplus \biggl\lfloor\frac{a_n}{2} \biggr\rfloor.
\end{equation*}
Throughout this paper, for each result we obtain for Halve Nim,
we will compare it with the corresponding result in Nim.

\section{$\P$-positions of Two-pile Games}\label{sec:TwoPile}

We begin by calculating the Grundy numbers of all two-pile games
$a^1 b^1$ with $0\le a,b \le 15$.
When a standard Nim move is applied to the game $a^1 b^1$,
the resulting position is either $x^1 b^1$, for $0\le x<a$, or
$a^1 y^1$, for $0\le y<b$.
When the halving operation is applied to $a^1 b^1$, the
resulting position is
$\big(\lfloor \tfrac{a}{2}\rfloor\big)^1
\big(\lfloor \tfrac{b}{2}\rfloor\big)^1$.
Thus,
\begin{align*}
    \Gr(a^1 b^1)=
    \mex\big( \{ \Gr(x^1 b^1): 0\le x<a\} &\cup
    \{ \Gr(a^1 y^1): 0\le y<b\} \\
    &\cup
    \{\Gr\big( \big(\lfloor \tfrac{a}{2}\rfloor\big)^1
\big(\lfloor \tfrac{b}{2}\rfloor\big)^1 \big)\}\big).
\end{align*}
The values of $\Gr(a^1 b^1)$, for  $0\le a,b \le 15$, are displayed in
Table \ref{tab:Grundyab0Alt}.
These values were calculated using a Maple program.

\begin{table}[ht]
\centering
{\small
\begin{tabular}{c|cccccccccccccccc}
$a{\backslash}b$ &0 &1  &2 &3  &4 &5  &6 &7  &8 &9  &10
 &11  &12 &13  &14 &15   \\
 \hline
 0  &0   &1 &  2 &  3 &  4 &  5 &  6 &  7 &  8 & 9
 &10 & 11 & 12 & 13 & 14 &15
      \\[0.038in]
 1  &1 &2 &  0 &  4 &  3 &  6 &  5 &  8 &  7 & 10 &  9 & 12 & 11 & 14 & 13 & 16
      \\[0.038in]
 2  &2 &0 &  1 &  5 &  6 &  3 &  7 &  9 &  4 &  8 & 11 & 10 & 13 & 12 & 15 & 14
      \\[0.038in]
 3  &3 &4 &  5 &  0 &  1 &  2 &  8 &  6 &  9 &  7 & 12 & 13 & 10 & 11 & 16 & 17
      \\[0.038in]
 4  &4 &3 &  6 &  1 &  0 &  7 &  2 & 10 &  5 & 11 &  8 &  9 & 14 & 15 & 12 & 13
      \\[0.038in]
 5  &5 &6 &  3 &  2 &  7 &  0 &  1 &  4 & 10 & 12 & 13 &  8 &  9 & 16 & 11 & 18
      \\[0.038in]
 6  &6 &5 &  7 &  8 &  2 &  1 &  3 & 11 &  0 &  4 & 14 & 15 & 16 &  9 & 10 & 12
      \\[0.038in]
 7  &7  &8 &  9 &  6 & 10 &  4 & 11 &  1 &  2 &  0 &  3 &  5 & 15 & 17 & 18 & 19
  \\[0.038in]
 8  &8  &7 &  4 &  9 &  5 & 10 &  0 &  2 &  1 &  3 &  6 & 14 & 17 & 18 & 19 & 11
  \\[0.038in]
 9  &9   &10 &  8 &  7 & 11 & 12 &  4 &  0 &  3 &  1 &  2 &  6 &  5 & 19 & 17 & 20
   \\[0.038in]
 10  &10   &9 & 11 & 12 &  8 & 13 & 14 &  3 &  6 &  2 &  1 &  4 &  0 &  5 &  7 & 21
  \\[0.038in]
 11  &11 &12 & 10 & 13 &  9 &  8 & 15 &  5 & 14 &  6 &  4 &  1 &  2 &  0 &  3 &  7
  \\[0.038in]
 12  &12 &11 & 13 & 10 & 14 &  9 & 16 & 15 & 17 &  5 &  0 &  2 &  1 &  4 &  6 &  3
  \\[0.038in]
 13  &13 &14 & 12 & 11 & 15 & 16 &  9 & 17 & 18 & 19 &  5 &  0 &  4 &  1 &  2 &  6
  \\[0.038in]
 14  &14 &13 & 15 & 16 & 12 & 11 & 10 & 18 & 19 & 17 &  7 &  3 &  6 &  2 &  0 &  4
  \\[0.038in]
 15  &15 &16 & 14 & 17 & 13 & 18 & 12 & 19 & 11 & 20 & 21 &  7 &  3 &  6 &  4 &  0
  \\[0.038in]
\end{tabular}
}
\caption{Grundy number of $a^1 b^1$, denoted by $\Gr(a^1 b^1)$,
is at the $(a,b)$-entry.}
\label{tab:Grundyab0Alt}
\end{table}

\begin{remark}\label{obs:Unique}
We observe that for a fixed
nonnegative integer $a$, if there
exists a nonnegative integer $b$
such that $a^1 b^1$ is a $\mathcal{P}$-position,
then $b$ is unique.
Suppose there exists $b$ such that $a^1 b^1$ is a
$\mathcal{P}$-position.
Let $b$ be the smallest nonnegative
integer such that
$a^1 b^1$ is a $\mathcal{P}$-position.
Since $a^1 b^1$ is reachable from $a^1 y^1$
for all $y>b$, $a^1 y^1$
is an $\mathcal{N}$-position
for all $y>b$.
This observation leads to the following definition.
\end{remark}
\vspace{0.2in}

\begin{definition}
    The
    \emph{Two-pile Halve Nim
    $\mathcal{P}$-position function}
    $\varphi_0:\mathbb{N} \to \mathbb{N}$
    returns the unique nonnegative integer
    for which
    $x^1 \; \bigl(\varphi_0(x)\bigr)^1$ is a $\mathcal{P}$-position.
\end{definition}

We make use of Remark \ref{rem:FindPposition} to determine
the two-pile games $a^1 b^1$ that are $\P$-positions.
Suppose we know $\varphi_0(x)$, for all $0\le x <a$,
and we want to find
$\varphi_0(a)$. If there exists $b<a$ such that $\varphi_0(b)=a$,
then $\varphi_0(a)=b$.
Otherwise, $\varphi_0(a)\ge a$.
From the list $a,a+1,a+2,\ldots$, we find the smallest value $b$
such that $\varphi_0(x)\ne b$, for all $0\le x<a$, and
$\varphi_0\big(\lfloor\tfrac{a}{2}\rfloor\big)
\ne \lfloor\tfrac{b}{2}\rfloor$. Then $\varphi_0(a)=b$.
Thus in Table \ref{tab:abPpoisitions},
we list the games $a^1 b^1$ that are
$\P$-positions, for $0\le a\le 79$.
The list of games $a^1 b^1$ that are $\P$-positions in Table
\ref{tab:abPpoisitions} and the list of games $a^1 b^1 c^1$,
for $0\le a,b\le 79$ and $1\le c\le 10$, in the tables throughout
this paper have been confirmed by Maple program calculations.

\begin{table}[ht]
\centering
\begin{tabular}{ccccccccccccc}
\multicolumn{8}{c}{List of games $a^1 b^1$ that are
$\mathcal{P}$-positions} \\
\hline
(0,0) &(1,2)  &(2,1)  &(3,3)  &(4,4)   &(5,5)
&(6,8)  &(7,9)  \\[0.025in]
(8,6)  &(9,7)  &(10,12)  &(11,13)  &(12,10)  &(13,11)  &(14,14)  &(15,15) \\[0.025in]
(16,16)  &(17,17)  &(18,18)  &(19,19) &(20,20) &(21,21)  &(22,22)  &(23,23)  \\[0.025in]
(24,24)   &(25,25) &(26,26)  &(27,27)  &(28,30)  &(29,31)
&(30,28)  &(31,29)  \\[0.025in]
(32,34)  &(33,35)  &(34,32)   &(35,33)
&(36,38)  &(37,39)  &(38,36)  &(39,37)
\\[0.025in]
(40,42) &(41,43)  &(42,40)  &(43,41)  &(44,46)   &(45,47)
&(46,44)  &(47,45)  \\[0.025in]
(48,50)  &(49,51)  &(50,48) &(51,49)  &(52,54)  &(53,55)  &(54,52)   &(55,53) \\[0.025in]
(56,56)  &(57,57)  &(58,58)  &(59,59) &(60,60)  &(61,61)  &(62,62)  &(63,63)  \\[0.025in]
(64,64)   &(65,65) &(66,66)  &(67,67)  &(68,68)  &(69,69)  &(70,70) &(71,71)  \\[0.025in]
(72,72)  &(73,73)  &(74,74)   &(75,75)
&(76,76)  &(77,77)  &(78,78)  &(79,79)
\\[0.025in]
\end{tabular}
\caption{$\mathcal{P}$-position games $a^1 b^1$
listed as $(a,b)$.}
\label{tab:abPpoisitions}
\end{table}

\begin{remark}\label{rem:BeginInduction0}
    The proof of Theorem \ref{thm:TwoHeapPposition} is by mathematical induction.
    In each step of the proof, 
    we want to show that $a^1 \big(\varphi_0(a)\big)^1$
    is a $\P$-position assuming that for all $x<a$,
    $x^1 y^1$ is a $\P$-position if and only if     
    $y=\varphi_0(x)$.
    We will observe that there exists a positive integer $a_0$
    such that for all $y<a_0$, $\varphi_0(y)<a_0$.
    Let $b=\varphi_0(a)$.
    Suppose Player 1 applies a standard Nim move to $a^1 b^1$.
    The resulting position is either 
    $y^1 b^1$ for $0\le y<a$ or $a^1 y^1$ for $0\le y<b$.
    Suppose Player 1 removes either at least $a-a_0+1$
    coins from the pile of height $a$ or at least $b-a_0+1$
    coins from the pile of height $b$ to obtain
    $y^1 z^1$ for $0\le y<a_0$ and $z\in\{a,\, b\}$.
    Then Player 2 removes $z-\varphi_0(y)$ coins from
    the pile of height $z$ to obtain $y^1 (\varphi_0(y))^1$,
    which is a $\P$-position by the inductive hypothesis.
    Thus, we may assume that the position after Player 1 applies a
    standard Nim move to $a^1 b^1$ is either
    $y^1 b^1$ for $a_0\le y<a$ or
    $a^1 y^1$ for $a_0\le y<b$.
\end{remark}

\begin{theorem}[\cite{CurranEtAl2026}, Theorem 1]\label{thm:TwoHeapPposition}
The first 14 values of $\varphi_0$, in the form of 
$(a,\varphi_0(a))$, are displayed in Table \ref{tab:abPpoisitions}.
     \begin{enumerate}
       \item Suppose $7\cdot 2^{2n-1} \le x <7\cdot 2^{2n}$
       for some positive integer $n$.
       Then
       \begin{equation*}
        \varphi_0(x)=x.
       \end{equation*}
    \item Suppose $7\cdot 2^{2n} \le x <7\cdot 2^{2n+1}$
    for some positive integer $n$.
       Then
       \begin{equation*}
         \varphi_0(x)= x \oplus 2.
       \end{equation*}
   \end{enumerate}
   Hence, $a^1 b^1$ is a $\mathcal{P}$-position if and only if
   $b=\varphi_0(a)$.
\end{theorem}

\begin{proof}
  We observe that $0^2$ is a $\mathcal{P}$-position.
Next, we show that $1^1 2^1$ is a $\P$-position.
Suppose Player 1 applies a standard Nim move to $1^1 2^1$.
By Remark \ref{rem:BeginInduction0},
we may assume that the resulting position is $1^2$.
Player 2 applies the halving operation to obtain $0^2$.
Suppose Player 1 applies the halving operation on $1^1 2^1$
to obtain $0^1 1^1$.
Player 2 applies the halving operation to obtain $0^2$.
Therefore,  $1^1 2^1$ is a $\P$-position.

We want to show that $x^2$ is a $\P$-position for $3\le x\le 5$.
Suppose Player 1 applies a standard Nim move to $x^2$.
By Remark \ref{rem:BeginInduction0},
we may assume that the resulting position is an $\N$-position.
Suppose Player 1 applies the halving operation to $x^2$
to obtain $y^2$ where $y\in\{1,\, 2\}$.
For $1^2$, Player 2 applies the halving operation
to obtain $0^2$.
For $2^2$, Player 2 removes 1 coin from a pile of height 2
to obtain $1^1 2^1$, which is a $\P$-position.
Therefore, $x^2$ is a $\P$-position.

We want to show that $(4x+2+r)^1 (4x+4+r)^1$
is a $\mathcal{P}$-position for $x\in\{1,\, 2\}$ and $r\in\{0,\, 1\}$.
We assume that for $y<4x+2$, $y^1 z^1$ is a 
$\mathcal{P}$-position if and only if $z=\varphi(y)$.
We observe that for all $y<4x+2$,
  $\varphi(y)<4x+2$.
Suppose Player 1 applies a standard Nim move to $(4x+2+r)^1 (4x+4+r)^1$.
By Remark \ref{rem:BeginInduction}, we may assume that 
the resulting position is
 either $y^1 (4x+4+r)^1$ for $4x+2\le y <4x+2+r$ or 
 $(4x+2+r)^1 y^1$ for $4x+2\le y <4x+4+r$.
  Suppose, for $y\in\{1,\, 2\}$, Player 1 removes $y+r$ coins from
  the pile of height $4x+4+r$ 
  to obtain $(4x+2+r)^1 (4x+4-y)^1$.
  Player 2 applies the halving operation to obtain $(2x+1)^2$, for $x\in\{1,\, 2\}$,
  which is a $\P$-position.
 Suppose Player 1 applies the halving operation 
  to $(4x+2+r)^1 (4x+4+r)^1$
  to obtain $(2x+1)^1 (2x+2)^1$.
  Player 2 removes 1 coin from the pile of height $2x+2$ to obtain $(2x+1)^2$,
  which is a $\P$-position.
Therefore, $(4x+2)^1 (4x+4)^1$ is a $\P$-position.

 Assume $r=1$.
 Suppose Player 1 removes 1 coin from the pile of
 height either $4x+3$ or $4x+5$ to obtain 
 $(4x+2)^1 (4x+5)^1$ or $(4x+3)^1 (4x+4)^1$.
 Player 2 removes 1 coin from the opposing pile to obtain
 $(4x+2)^1 (4x+4)^1$, which is a $\mathcal{P}$-position.
 Therefore, $(4x+3)^1 (4x+5)^1$ is a $\mathcal{P}$-position.
 
Suppose $7\cdot 2^{2n-1}\le x< 7\cdot 2^{2n}$ for some 
positive integer $n$ and,
for all nonnegative integers $y<x$,
  $y^1 z^1$ is a $\mathcal{P}$-position if and 
  only if $z=\varphi(y)$.
We need to establish that $x^2$ is 
a $\mathcal{P}$-position.
 We observe that $\varphi(y)<x$ for all $y <x$.
Suppose Player 1 applies a standard Nim move to $x^2$.
By Remark \ref{rem:BeginInduction0},
we may assume that the resulting position is an $\N$-position.
   
  Suppose Player 1 applies the halving operation 
  to $x^2$ to obtain $y^2$ where $y=\lfloor \tfrac{x}{2}\rfloor$.
  Except for $y\in\{12,\, 13\}$, Player 2 applies the halving operation
  to obtain $\lfloor \tfrac{y}{2} \rfloor$,
   which is a $\P$-position since either 
  $3\le \lfloor \tfrac{y}{2} \rfloor \le 5$ or $7\cdot 2^{2n-3}\le \lfloor \tfrac{y}{2} \rfloor<7\cdot 2^{2n-2}$.
  If $y\in\{12,\, 13\}$, Player 2 removes 2 coins from one of the piles of height $y$
  to obtain $(y-2)^1 y^1$, which is a $\P$-position.
  Therefore, $x^2$ is a $\mathcal{P}$-position.
  
Suppose $7\cdot 2^{2n}\le 4x< 7\cdot 2^{2n+1}$ for some positive
integer $n$ and,
for all nonnegative integers $y<4x$,
  $y^1 z^1$ is a $\mathcal{P}$-position if and 
  only if $z=\varphi(y)$.
We want to show that $(4x+r)^1 (4x+2+r)^1$
is a $\mathcal{P}$-position for $r\in\{0,\, 1\}$.
We observe that for all $y<4x$, $\varphi(y)<4x$.
Suppose Player 1 applies a standard Nim move to $(4x+r)^1 (4x+2+r)^1$.
By Remark \ref{rem:BeginInduction0}, we may assume that
the resulting position is
 either $y^1 (4x+2+r)^1$ for $4x\le y <4x+r$ or
 $(4x+r)^1 y^1$ for $4x\le y <4x+2+r$.

 Assume $r\in\{0,\, 1\}$.
  Suppose, for $y\in\{1,\, 2\}$, Player 1 removes $y+r$ coins from the 
  pile of height $4x+2+r$ to obtain 
  $(4x+r)^1 (4x+2-y)^1$.
  Player 2 applies the halving operation
  to obtain $(2x)^2$,
  which is a $\mathcal{P}$-position
 since $7\cdot 2^{2n-1} \le 2x < 7\cdot 2^{2n}$.
  Suppose Player 1 applies the halving operation to 
  $(4x+r)^1 (4x+2+r)^1$ to obtain
  $(2x)^1 (2x+1)^1$.
  Player 2 removes 1 coin from the pile of height $2x+1$ to obtain $(2x)^2$,
  which is a $\P$-position.
  Hence, $(4x)^1 (4x+2)^1$ is a $\mathcal{P}$-position.

  Assume $r=1$.
Suppose Player 1 removes 1 coin from the pile of height 
either $4x+1$ or $4x+3$
to obtain $(4x)^1 (4x+3)^1$ or $(4x+1)^1 (4x+2)^1$.
Player 2 removes 1 coin from the opposing pile to
obtain $(4x)^1 (4x+2)^1$, which is a 
$\mathcal{P}$-position.
  Therefore, $(4x+1)^1 (4x+3)^1$ is a $\mathcal{P}$-position.
\end{proof}

For $a,b\ge 14$, we may restate Theorem \ref{thm:TwoHeapPposition} in terms of
$a\oplus b$.

\begin{corollary} \label{cor:TwoPileSecond}
    Suppose $a,b\ge 14$. Then, $a^1 b^1$ is a $\P$-position if and
    if
    \begin{equation}\label{eqn:0PileResult}
       a\oplus b= 1 + (-1)^{\lfloor \log_2(a/7)\rfloor}.
    \end{equation}
\end{corollary}

\begin{proof}
    Let $m$ be the positive integer such that $7\cdot 2^m\le a<7\cdot 2^{m+1}$.
If $m$ is odd, $1 + (-1)^{\lfloor \log_2(a/7)\rfloor}=0$.
If $m$ is even, $1 + (-1)^{\lfloor \log_2(a/7)\rfloor}=2$.
By Theorem \ref{thm:TwoHeapPposition}, $a^1 b^1$ is a $\P$-position if and only if $a\oplus b=1 + (-1)^{\lfloor \log_2(a/7)\rfloor}$.
\end{proof}

\begin{remark}
    For comparison, in Nim, the game $a^1 b^1$ is a $\P$-position
    if and only if
    \begin{equation*}
       a\oplus b= 0.
    \end{equation*}
\end{remark}

\section{$\P$-positions of Three-pile Games}\label{sec:ThreePile}

In this section
we consider the three-pile games $a^1\, b^1\, c^1$
of Halve Nim, whenever $1\le c\le 10$.
The main result of Section \ref{sec:ThreePile} is 
Theorem \ref{thm:MainResult} which characterizes
the $\P$-position of the
three-pile game of Halve Nim $a^1 b^1 c^1$
in terms of Nim addition, for $a,b\ge 30$ and $0\le c\le 7$.

\begin{theorem}[Main Result]\label{thm:MainResult}
    Suppose $a$, $b$, and $c$ are integers such that $a,b\ge 30$.
    \begin{enumerate}
        \item Suppose $0\le c \le 3$. Then $a^1 b^1 c^1$ is a $\P$-position 
        if and only if
        \begin{equation}\label{eqn:FirstBitsum}
            a\oplus b = \tfrac{1}{2}\big( 3 +(-1)^{c+1}\big) +(-1)^{\lfloor \log_2(a/7)\rfloor +\big(1+(-1)^{\lfloor c/2\rfloor+1}\big)/2}.
         \end{equation}
            \item Suppose $4\le c \le 7$. Then $a^1 b^1 c^1$ is a $\P$-position 
        if and only if
        \begin{equation}\label{eqn:SecondBitsum}
            (a+2)\oplus (b+2) =c-2.
        \end{equation}
    \end{enumerate}
\end{theorem}

\begin{remark}
    For comparison, in Nim, the game $a^1 b^1 c^1$ is a $\P$-position
    if and only if
    \begin{equation*}
       a\oplus b= c.
    \end{equation*}
\end{remark}

In order to prove \eqref{eqn:FirstBitsum} in Theorem \ref{thm:MainResult},
it will be helpful to define the function appearing on the right hand side of \eqref{eqn:FirstBitsum}.

\begin{definition}
    Let $x$ and $c$ be integers such that $x\ge 1$ and $0\le c\le 3$.
    The \emph{$\P$-position value function} $\gamma_c:\mathbb{Z}_{+}\to\mathbb{Z}$ is defined by
    \begin{equation*}
         \gamma_c(x)    = \tfrac{1}{2}\big( 3 +(-1)^{c+1}\big) +(-1)^{\lfloor \log_2(x/7)\rfloor +\big(1+(-1)^{\lfloor c/2\rfloor+1}\big)/2}.
         \end{equation*}
\end{definition}

The remainder of Section \ref{sec:ThreePile} is dedicated to proving
Theorem \ref{thm:MainResult} and attempting to determine
if there is a similar Nim-like characterization of the
$\P$-positions of the game $a^1 b^1 c^1$ for $8\le c\le 10$.
We begin the proof of Theorem \ref{thm:MainResult}
with Remark \ref{rem:ThreePilePposition}.

\begin{remark}\label{rem:ThreePilePposition}
Consider all games $a^1 b^1 c^1$ where $c$ is fixed.
We observe that for a fixed 
nonnegative integer $a$, if there exists a nonnegative integer $b$
such that $a^1 b^1 c^1$ is a $\mathcal{P}$-position,
then $b$ is unique.
Suppose there exists $b$ such that 
$a^1 b^1 c^1$ is a 
$\mathcal{P}$-position.
Let $b$ be the smallest nonnegative 
integer such that
$a^1 b^1 c^1$ is a $\mathcal{P}$-position.
Since $a^1 b^1 c^1$ is reachable from $a^1 x^1 c^1$
for all $x>b$, $a^1 x^1 c^1$ 
is an $\mathcal{N}$-position
for all $x>b$. 
\end{remark}

This observation from Remark \ref{rem:ThreePilePposition} 
leads to Definition \ref{defn:ThreeHeap}.

\begin{definition}\label{defn:ThreeHeap}
   Let $c$ be a positive integer.
    The 
    {\em Three-pile with one pile of height $c$ function}
    $\varphi_c:\mathbb{N}  \to \mathbb{N}$
    is defined such that
    $\varphi_c(x)$ is the unique nonnegative integer
    for which
    $x^1 \; \bigl(\varphi_c(x)\bigr)^1 c^1$ is a $\mathcal{P}$-position.
\end{definition}

The $\mathcal{P}$-positions of the three-pile
games $a^1 \, b^1 \, 1^1$, for $0\le a \le 79$,
are listed in Table \ref{tab:ab1Ppoisitions}
as the ordered pair $(a,b)$.
The results in Table \ref{tab:ab1Ppoisitions}
as well as the $\P$-position calculations of
all three-pile games $a^1 b^1 c^1$
throughout this section
were confirmed by a Maple program calculation.

\begin{remark}
We observe that the function
    $\varphi_c(x)$ is an involution.
    In the proofs of results involving 
    $\varphi_c(x)$, whenever we establish that
    $y=\varphi_c(x)$, it will also be true that $x=\varphi_c(y)$.
    So, we will accept that this as true without explicitly stating 
    this fact.
\end{remark}

\begin{remark}
For each result on the 
$\P$-positions of the 
three-pile game
$a^1 \, b^1 \, c^1$ of Halve Nim in this section of the paper,
for $1\le c\le 10$, there is a threshold value
$b_0$ such that
for $a,b\ge b_0$, either the  $\mathcal{P}$-positions of
$a^1 \, b^1 \, c^1$ depend on
whether $7\cdot 2^{2n-1}\le a,b < 7\cdot 2^{2n}$
or $7\cdot 2^{2n}\le a,b < 7\cdot 2^{2n+1}$
for some positive integer $n$,
or the  $\mathcal{P}$-positions of
$a^1 \, b^1 \, c^1$ are periodic.
This threshold value occurs at $b_0=14$ for $c=1$,
$b_0=30$ for $c\in\{2,\,3,\, 4,\, 5,\, 6,\, ,7\}$, 
$b_0=64$ for $c\in\{8,\, 9 \}$, and $b_0=66$ for $c=10$.
\end{remark}

\begin{table}[ht]
\centering
\begin{tabular}{ccccccccccccccc}
\multicolumn{10}{c}{List of games $a^1 b^1 1^1$ that are
$\mathcal{P}$-positions} \\
\hline
(0,2) &(1,3)  &(2,0)  &(3,1)  &(4,5)   &(5,4)
&(6,9)  &(7,8)  &(8,7)  &(9,6)  
\\[0.025in]
(10,13) &(11,12)  &(12,11)  &(13,10)  &(14,15)  &(15,14)
&(16,17)  &(17,16)  &(18,19)  &(19,18)  
\\[0.025in]
(20,21) &(21,20)  &(22,23)  &(23,22)  &(24,25)   &(25,24)
&(26,27)  &(27,26)  &(28,31)  &(29,30)  
\\[0.025in]
(30,29) &(31,28)  &(32,35)  &(33,34)  &(34,33)   &(35,32)
&(36,39)  &(37,38)  &(38,37)  &(39,36)  
\\[0.025in]
(40,43) &(41,42)  &(42,41)  &(43,40)  &(44,47)   &(45,46)
&(46,45)  &(47,44)  &(48,51)  &(49,50)  
\\[0.025in]
(50,49) &(51,48)  &(52,55)  &(53,54)  &(54,53)   &(55,52)
&(56,57)  &(57,56)  &(58,59)  &(59,58)  
\\[0.025in]
(60,61) &(61,60)  &(62,63)  &(63,62)  &(64,65)   &(65,64)
&(66,67)  &(67,66)  &(68,69)  &(69,68)  
\\[0.025in]
(70,71) &(71,70)  &(72,73)  &(73,72)  &(74,75)   &(75,74)
&(76,77)  &(77,76)  &(78,79)  &(79,78) 
\\[0.025in]
\end{tabular} 
\caption{$\mathcal{P}$-position games $a^1 b^1 1^1$ 
listed as $(a,b)$.}
\label{tab:ab1Ppoisitions}
\end{table}

The list of games $a^1 b^1 1^1$ that are $\P$-positions in Table
\ref{tab:ab1Ppoisitions} and the list of games $a^1 b^1 c^1$,
for $0\le a,b\le 79$ and $2\le c\le 10$, in the tables throughout
this section have been confirmed by Maple program calculations.
We characterize the three-pile games $a^1 b^1 1^1$ that
are $\mathcal{P}$-positions in Theorem \ref{thm:ThreeHeapab1Pposition}.

\begin{theorem}\label{thm:ThreeHeapab1Pposition}
The first 14 values of $\varphi_1$, in the form of 
$(a,\varphi_1(a))$, are displayed in Table \ref{tab:ab1Ppoisitions}.

   \begin{enumerate}
       \item Suppose $7\cdot 2^{2n-1} \le x < 7\cdot 2^{2n}$ 
       for some integer $n\ge 1$.
       Then 
       \begin{equation*}
        \varphi_1(x)= x\oplus 1.   
       \end{equation*}
        \item Suppose $7\cdot 2^{2n} \le x <7\cdot 2^{2n+1}$ 
    for some integer $n\ge 1$.
       Then 
       \begin{equation*}
         \varphi_1(x)= x \oplus 3.  
       \end{equation*}
   \end{enumerate}
   Hence, $a^1 b^1 1^1$ is a $\mathcal{P}$-position if and only if
   $b=\varphi_1(a)$.
   Furthermore, for $a,b\ge 14$,
   $a^1 b^1 1^1$ is a $\P$-position if and only if
    \begin{equation}\label{eqn:1PileResult}
       a\oplus b= 2 + (-1)^{\lfloor \log_2(a/7)\rfloor}.
    \end{equation}
 \end{theorem}

\begin{remark}
    We observe that $\varphi_1(x)=\varphi_0(x)\oplus 1$
    for $x\ge 6$.
\end{remark}

\begin{proof}
By Theorem \ref{thm:TwoHeapPposition}, $0^1 2^1 1^1$ is a $\P$-position.
We show that $1^2 3^1$ is a $\P$-position.
Suppose Player 1 removes 1 coin from a pile of height either 1 or 3
to obtain $1^1 3^1$ or $1^2 2^1$.
Player 2 removes 1 coin from the opposing pile to obtain $1^1 2^1$,
which is a $\P$-position by Theorem \ref{thm:TwoHeapPposition}.
Suppose, from $y\in\{2,\, 3\}$, Player 1 removes $y$ coins from the pile of height 3
to obtain $1^2 (3-y)^1$.
Player 2 applies the halving operation to $1^2 (3-y)^1$ to obtain $0^1$.
Suppose Player 1 applies the halving operation to $1^2 3^1$ to obtain $1^1$.
Player 2 removes 1 coin from the pile of height 1 to obtain $0^1$.
Thus,  $1^2 3^1$ is a $\P$-position.

We want to show that $4^1 5^1 1^1$ is a $\P$-position.
 We observe that for all $y <4$, $\varphi_1(y)<2x$.
Suppose Player 1 applies a standard Nim move to $4^1 5^1 1^1$.
By Remark \ref{rem:BeginInduction}, we may assume that
  the resulting position is either
    $4^2 1^1$ or $4^1 5^1$.
    Suppose Player 1 removes 1 coin from the pile of height of either 5 or 1
    to obtain $4^2 1^1$ or $4^1 5^1$.
    Player 2 removes 21 coin fro  the opposing pile to obtain $4^2$,
    which is a $\P$-position by Theorem \ref{thm:TwoHeapPposition}.
Suppose Player 1 applies the halving operation to $4^1 5^1 1^1$ to obtain $2^2$.
Player 2 removes 1 coin from a pile of height 2 to obtain $1^1 2^1$,
which is a $\P$-position by Theorem \ref{thm:TwoHeapPposition}.
Thus,  $4^1 5^1 1^1$ is a $\P$-position.

 We want to show that $(4x+2+r)^1 (4x+5-r)^1 1^1$ is
  $\mathcal{P}$-position for $x\in\{1,\, 2\}$ and $r\in\{0,1\}$.
We observe that for all $y<4x+2$,
  $\varphi_1(y)<4x+2$.
  Suppose Player 1 applies a standard Nim move to $(4x+2+r)^1 (4x+5-r)^1 1^1$.
By Remark \ref{rem:BeginInduction}, we may assume that
  the resulting position is either
 $y^1 (4x+5-r)^1 1^1$ for $4x+2\le y <4x+2+r$,
or $(4x+2+r)^1 y^1 1^1$ for $4x+2\le y <4x+5-r$, 
  or $(4x+2+r)^1 (4x+5-r)^1$.

Assume $r=0$. Suppose Player 1 removes 1 coin from
the pile of height either $4x+5$ or $1$ to 
obtain $(4x+2)^1 (4x+4)^1 1^1$ or $(4x+2)^1 (4x+5)^1$.
Player 2 removes $1$ coin from the opposing pile
to obtain $(4x+2)^1 (4x+4)^1$, which is a $\mathcal{P}$-position
by Theorem \ref{thm:TwoHeapPposition}.

Assume $r\in\{0,\, 1\}$.
Suppose, for $y\in\{2,\, 3\}$, Player 1 removes $y-r$ coins from
the pile of height $4x+5-r$ to obtain $(4x+2+r)^1 (4x+5-y)^1 1^1$.
Player 2 applies the halving operation to obtain $(2x+1)^2$,
which is a $\mathcal{P}$-position
by Theorem \ref{thm:TwoHeapPposition}.
  Suppose Player 1 applies the halving operation to  $(4x+2+r)^1 (4x+5-r)^1 1^1$ to obtain $(2x+1)^1 (2x+2)^1$.
  Player 2 removes $1$ coin from the pile of height $2x+2$ to obtain $(2x+1)^2$,
which is a $\mathcal{P}$-position
by Theorem \ref{thm:TwoHeapPposition}.
Hence, $(4x+2)^1 (4x+5)^1 1^1$ is a $\mathcal{P}$-position.

Assume $r=1$.
Suppose Player 1 removes $1$ coin from
the pile of height either $4x+3$ or $1$ to obtain $(4x+2)^1 (4x+4)^1 1^1$ or $(4x+3)^1 (4x+4)^1$.
Player 2 removes $1$ coin from the opposing pile to obtain $(4x+2)^1 (4x+4)^1$,
which is a $\mathcal{P}$-position by Theorem \ref{thm:TwoHeapPposition}.
Therefore, $(4x+3)^1 (4x+4)^1 1^1$ is a $\mathcal{P}$-position.

Suppose $7\cdot 2^{2n-1} \le 2x < 7\cdot 2^{2n}$ for some 
    integer 
    $n\ge 1$ and, for all nonnegative integers $y<2x$,
  $y^1 z^1 1^1$ is a $\mathcal{P}$-position if and 
  only if $z=\varphi_1(y)$.
  We need to establish that $(2x)^1 (2x+1)^1 1^1$ is a 
  $\mathcal{P}$-position.
 After applying a standard Nim move to 
 $(2x)^1 (2x+1)^1 1^1$, the resulting position is
 either $y^1 (2x+1)^1 1^1$ for
  $0\le y <2x$, $(2x)^1 y^1 1^1$ for $0\le y <2x+1$,
  or $(2x)^1 (2x+1)^1$.
  We observe that for all $y<2x$,
  $\varphi_1(y)< 2x$.
  
  Suppose that Player 1 removes either at least 1 coin from the pile of height $2x$
  or at least 2 coins from the pile of height $2x+1$ to obtain $y^1 z^1 1^1$, for
  $0\le y <2x$ and $z\in\{2x,\, 2x+1\}$.
  Player 2 removes $z-\varphi_1(y)$ coins from the 
  pile of height $z$ to obtain 
  $y^1 (\varphi_1(y))^1 1^1$,
  which is a $\mathcal{P}$-position 
  by the inductive hypothesis.
  Suppose Player $1$ removes 1 coin from the pile
  of height either $2x+1$ or $1$ to obtain $(2x)^2 1^1$
  or $(2x)^1 (2x+1)^1$.
  Player 2 removes 1 coin
  from the opposing pile to obtain $(2x)^2$,
  which is a $\mathcal{P}$-position
  by Theorem~\ref{thm:TwoHeapPposition}.
Suppose Player 1 applies the halving operation to
$(2x)^1 (2x+1)^1 1^1$ to obtain $(x)^2$ 
  where $7\cdot 2^{2n-2} \le x <  7\cdot 2^{2n-1}$.
  Except for $x\in\{12,13\}$,
  Player 2 applies the halving operation to obtain  
  the position $\bigl(\lfloor\tfrac{x}{2}\rfloor\bigr)^2$,
  where either $3\le  \lfloor \tfrac{x}{2}\rfloor  \le 5$ or
  $7\cdot 2^{2n-3} \le \lfloor \tfrac{x}{2}\rfloor <  7\cdot 2^{2n-2}$ for $n\ge 2$,
  which is a $\mathcal{P}$-position 
 by Theorem~\ref{thm:TwoHeapPposition}.
  When $x\in\{12,13\}$, Player 2 removes 2 coins from one of
  the piles of height $x$ to obtain $(x-2)^1 x^1$,
  which is listed as a $\P$-position in Table \ref{tab:abPpoisitions}.
Therefore, $(2x)^1 (2x+1)^1 1^1$ 
is a $\mathcal{P}$-position.

Suppose $7\cdot 2^{2n} \le 4x < 7\cdot 2^{2n+1}$ for some 
    integer 
    $n\ge 1$ and, for all nonnegative integers $y<4x$,
  $y^1 z^1 1^1$ is a $\mathcal{P}$-position if and 
  only if $z=\varphi_1(y)$.
  We need to establish that $(4x+r)^1 (4x+3-r)^1 1^1$ is
  $\mathcal{P}$-position for $r\in\{0,1\}$.
 After applying a standard Nim move to 
 $(4x+r)^1 (4x+3-r)^1 1^1$, the resulting position is
 either $y^1 (4x+3-r)^1 1^1$ for
  $0\le y <4x+r$, $(4x+r)^1 y^1 1^1$ for $0\le y <4x+3-r$,
  or $(4x+r)^1 (4x+3-r)^1$.
  We observe that for all $y<4x$,
  $\varphi_1(y)< 4x$.
  Suppose that Player 1 removes at least $r+1$ coins from
  the pile of height $4x+r$ or at least $4-r$ coins
  from the pile of height $4x+3-r$ to obtain
  $y^1 z^1 1^1$, for
  $0\le y <4x$ and $z\in\{4x+r,\, 4x+3-r\}$.
  Player 2 removes $z-\varphi_1(y)$ coins from the 
  pile of height $z$ to obtain 
  $y^1 (\varphi_1(y))^1 1^1$,
  which is a $\mathcal{P}$-position 
  by the inductive hypothesis.
  
Assume $r=0$. Suppose Player 1 removes 1 coin from
the pile of height either $4x+3$ or $1$ to 
obtain $(4x)^1 (4x+2)^1 1^1$ or $(4x)^1 (4x+3)^1$.
Player 2 removes $1$ coin from the opposing pile
to obtain $(4x)^1 (4x+2)^1$, which is a $\mathcal{P}$-position
by Theorem \ref{thm:TwoHeapPposition}.

Assume $r\in\{0,\, 1\}$.
Suppose, for $y\in\{2,\, 3\}$, Player 1 removes $y-r$ coins from
the pile of height $4x+3-r$ to obtain $(4x+r)^1 (4x+3-y)^1 1^1$.
Player 2 applies the halving operation to obtain $(2x)^2$,
which is a $\mathcal{P}$-position
by Theorem \ref{thm:TwoHeapPposition}.
  Suppose Player 1 applies the halving operation to  $(4x+r)^1 (4x+3-r)^1 1^1$ to obtain $(2x)^1 (2x+1)^1$.
  Player 2 removes $1$ coin from the pile of height $2x+1$ to obtain $(2x)^2$,
which is a $\mathcal{P}$-position
by Theorem \ref{thm:TwoHeapPposition}.
Hence, $(4x)^1 (4x+3)^1 1^1$ is a $\mathcal{P}$-position.

Finally, assume $r=1$.
Suppose Player 1 removes $1$ coin from
the pile of height either $4x+1$ or $1$ to obtain $(4x)^1 (4x+2)^1 1^1$ or $(4x+1)^1 (4x+2)^1$.
Player 2 removes $1$ coin from the opposing pile to obtain $(4x)^1 (4x+2)^1$,
which is a $\mathcal{P}$-position.
Therefore, $(4x+1)^1 (4x+2)^1 1^1$ is a $\mathcal{P}$-position.

Suppose $a,b\ge 14$.
Let $m$ be the positive integer such that $7\cdot 2^m\le a<7\cdot 2^{m+1}$.
If $m$ is odd, $\gamma_1(a)=1$.
If $m$ is even, $\gamma_1(a)=3$.
Thus, $a^1 b^1 1^1$ is a $\P$-position if and only if $a\oplus b=\gamma_1(a)$.\qed
\end{proof}

\begin{remark}\label{rem:BeginInduction}
    The proofs of our results on $\varphi_c(x)$ are by mathematical induction.
    We first establish the base case.
    There will be a threshold value $b_0$ for which the 
    behavior of the $\P$-positions of $a^1 b^1 c^1$ has regular
    behavior for $a,b\ge b_0$.
    The games $a^1 b^1 c^1$ that are $\P$-positions
    for $a,b<b_0$ will be listed in a table in the
    form of $(a,b)$. We will leave the details of
    establishing the base case to the reader.

    For the inductive step, we want to show that $a^1 \big(\varphi_c(a)\big)^1 c^1$
    is a $\P$-position assuming that for all $x<a$,
    $x^1 y^1 c^1$ is a $\P$-position if and only if     
    $y=\varphi_c(x)$.
    We will observe that there exists a positive integer $a_0$
    such that for all $y<a_0$, $\varphi_c(y)<a_0$.
    Let $b=\varphi_c(a)$.
    Suppose Player 1 applies a standard Nim move on $a^1 b^1 c^1$.
    The resulting position is either 
    $y^1 b^1 c^1$ for $0\le y<a$, $a^1 y^1 c^1$ for $0\le y<b$,
    or $a^1 b^1 y^1$ for $0\le y<c$.
    Suppose Player 1 removes either at least $a-a_0+1$
    coins from the pile of height $a$ or at least $b-a_0+1$
    coins from the pile of height $b$ to obtain
    $y^1 z^1 c^1$ for $0\le y<a_0$ and $z\in\{a,\, b\}$.
    Then Player 2 removes $z-\varphi_c(y)$ coins from
    the pile of height $z$ to obtain $y^1 (\varphi_c(y))^1 c^1$,
    which is a $\P$-position by the inductive hypothesis.
    Thus, we may assume that the position after Player 1 applies a
    standard Nim move to $a^1 b^1 c^1$ is either
    $y^1 b^1 c^1$ for $a_0\le y<a$,
    $a^1 y^1 c^1$ for $a_0\le y<b$,
    or $a^1 b^1 y^1$ for $0\le y<c$.
\end{remark}

\begin{table}[t]
\centering
\begin{tabular}{ccccccccccccccc}
\multicolumn{10}{c}{List of games $a^1 b^1 2^1$ that are
$\mathcal{P}$-positions} \\
\hline
(0,1) &(1,0)  &(2,2)  &(3,4)  &(4,3)   &(5,6)
&(6,5)  &(7,7)  &(8,8)  &(9,9)  
\\[0.025in]
(10,10) &(11,11)  &(12,12)  &(13,13)  &(14,18)  &(15,19)
&(16,20)  &(17,21)  &(18,14)  &(19,15)  
\\[0.025in]
(20,16) &(21,17)  &(22,26)  &(23,27)  &(24,28)   &(25,29)
&(26,22)  &(27,23)  &(28,24)  &(29,25)  
\\[0.025in]
(30,30) &(31,31)  &(32,32)  &(33,33)  &(34,34)   &(35,35)
&(36,36)  &(37,37)  &(38,38)  &(39,39)  
\\[0.025in]
(40,40) &(41,41)  &(42,42)  &(43,43)  &(44,44)   &(45,45)
&(46,46)  &(47,47)  &(48,48)  &(49,49)  
\\[0.025in]
(50,50) &(51,51)  &(52,52)  &(53,53)  &(54,54)   &(55,55)
&(56,58)  &(57,59)  &(58,56)  &(59,57)  
\\[0.025in]
(60,62) &(61,63)  &(62,60)  &(63,61)  &(64,66)   &(65,67)
&(66,64)  &(67,65)  &(68,70)  &(69,71)  
\\[0.025in]
(70,68) &(71,69)  &(72,74)  &(73,75)  &(74,72)   &(75,73)
&(76,78)  &(77,79)  &(78,76)  &(79,77)  
\\[0.025in]
\end{tabular} 
\caption{$\mathcal{P}$-position games $a^1 b^1 2^1$ 
listed as $(a,b)$.}
\label{tab:ab2Ppoisitions}
\end{table}

The $\mathcal{P}$-positions of the three-pile
games $a^1 \, b^1 \, 2^1$, for $0\le a \le 79$,
are listed in Table \ref{tab:ab2Ppoisitions}
as the ordered pair $(a,b)$.
We characterize the three-pile games $a^1 b^1 2^1$ that
are $\mathcal{P}$-positions in Theorem \ref{thm:ThreeHeapab2Pposition}.
In the next two theorems, we need to verify that the game
$(2x)^1 (2x+1)^1 1^1$  is an $\N$-position when 
$7\cdot 2^{2n} \le 2x < 7\cdot 2^{2n+1}$ 
    for some positive integer $n$. 
    
\begin{lemma}\label{lem:HalveCase0}
    Suppose $x$ is an integer such that $7\cdot 2^{2n} \le 2x < 7\cdot 2^{2n+1}$ 
    for some positive integer $n$. 
    Then the game $(2x)^1 (2x+1)^1 1^1$  is an $\N$-position.
\end{lemma}

\begin{proof}
  Since $\varphi_1(2x)=(2x)\oplus 3\in\{2x-1,2x+3\}$,
  $\varphi_1(2x)\ne 2x+1$.
  Thus, by Theorem~\ref{thm:ThreeHeapab1Pposition},
  $(2x)^1 (2x+1)^1 1^1$ is an $\mathcal{N}$-position.
\end{proof}

\begin{theorem}\label{thm:ThreeHeapab2Pposition}
The first 30 values of $\varphi_2$, in the form of 
$(a,\varphi_2(a))$, are displayed in Table \ref{tab:ab2Ppoisitions}.
   Let $x$ be a positive integer. 
   \begin{enumerate}
  \item Suppose either $30\le x<56$ or $7\cdot 2^{2n}\le x < 7\cdot 2^{2n+1}$ for some integer $n\ge 2$.
       Then 
       \begin{equation*}
         \varphi_2(x)=x.  
       \end{equation*}
       \item Suppose $7\cdot 2^{2n-1} \le x < 7\cdot 2^{2n}$ 
       for some integer $n\ge 2$.
       Then 
       \begin{equation*}
        \varphi_2(x)= x \oplus 2.   
       \end{equation*}
   \end{enumerate}
   Hence, $a^1 b^1 2^1$ is a $\mathcal{P}$-position if and only if
   $b=\varphi_2(a)$.
    Furthermore, for $a,b\ge 30$,
   $a^1 b^1 2^1$ is a $\P$-position if and only if
    \begin{equation}\label{eqn:2PileResult}
       a\oplus b= 1 + (-1)^{\lfloor \log_2(a/7)\rfloor+1}.
    \end{equation}
\end{theorem}

\begin{remark}
   We observe that $\varphi_2(x)=\varphi_0(x)\oplus 2$
    for $x\ge 30$.
\end{remark}

\begin{proof}
By Theorem \ref{thm:TwoHeapPposition},
$0^1 1^1 2^1$ is a $\P$-position.
We show $2^3$ is a $\P$-position.
Suppose, for $y\in\{1,\, 2\}$, Player 1 removes $y$ coins from 
a pile of height 2 to obtain $2^2 (2-y)^1$.
Player 2 removes $3-y$ coins from a pile of height 2 to obtain $1^1 2^2$, 
which is a $\P$-position by Theorem \ref{thm:TwoHeapPposition}.
Suppose Player 1 applies the halving operation to $2^3$
to obtain $1^3$.
Player 2 applies the halving operation to $1^3$ to obtain $0^1$.
Thus, $2^3$ is a $\P$-position.

We show $(2x+1)^1 (2x+2)^1 2^1$ is a $\P$-position for $x\in\{1,\, 2\}$.
We observe that for all $y<2x+1$, $\varphi_2(y)<2x+1$.
Suppose Player 1 applies a standard Nim move to $(2x+1)^1 (2x+2)^1 2^1$.
By Remark \ref{rem:BeginInduction}, we may assume that
the resulting position is
 either  $(2x+1)^2 2^1$ or $(2x+1)^1 (2x+2)^1 y^1$ for $0\le y <2$.
Suppose Player 1 either removes 1 coin from the pile of height $2x+2$
or removes $2$ coins from the pile of height 2 to obtain
$(2x+1)^1 2^1$ or $(2x+1)^1 (2x+2)^1$.
Player 2 applies the opposing move to obtain $(2x+1)^2$,
which is a $\P$-position by Theorem \ref{thm:TwoHeapPposition}.
Suppose Player 1 removes 1 coin from the pile of height 1
to obtain $(2x+1)^1 (2x+2)^1 1^1$.
Player 2 removes $3-x$ coins from a pile of height $2x+1$
to obtain $(2x+1)^1 (3x-2)^1 1^1$, which is a $\P$-position
by Theorem \ref{thm:ThreeHeapab1Pposition}.
Suppose Player 1 applies the halving operation to $(2x+1)^1 (2x+2)^1 2^1$
to obtain $x^1 (x+1)^1 1^1$.
Player 2 removes 1 coin from the pile of height $x$ 
to obtain $(x-1)^1 (x+1)^1 1^1$,
which is a $\P$-position by theorems 
\ref{thm:TwoHeapPposition} and \ref{thm:ThreeHeapab1Pposition}.
Thus,  $(2x+1)^1 (2x+2)^1 2^1$ is a $\P$-position.

We show $x^2 2^1$ is a $\P$-position for $7\le x\le 13$.
We observe that for all $y<x$, $\varphi_2(y)<x$.
Suppose Player 1 applies a standard Nim move to $x^2 2^1$.
By Remark \ref{rem:BeginInduction}, we may assume that
the resulting position is
 either  $x^2 y^1$ for $0\le y <2$.
Suppose Player 1 removes 1 coin from the pile of height 2 to obtain $x^2 1^1$.
If $7\le x\le 11$, Player 2 applies the halving operation to $x^2 1^1$
to obtain $y^2$ where $3\le y le 5$, which
is a $\P$-position by Theorem \ref{thm:TwoHeapPposition}.
If $x\in\{12,\, 13\}$, Player 2 removes $2x-23$ coins from a pile of 
height $x$ to obtain $(23-x)^2 x^1 1^1$, which is a
$\P$-position by Theorem \ref{thm:ThreeHeapab1Pposition}.
Suppose Player 1 removes 2 coins from the pile of height 2 to obtain $x^2$.
If $7\le x\le 11$, Player 2 applies the halving operation to $x^2$
to obtain $y^2$ where $3\le y le 5$, which
is a $\P$-position by Theorem \ref{thm:TwoHeapPposition}.
If $x\in\{12,\, 13\}$, Player 2 removes $2$ coins from a pile of 
height $x$ to obtain $(x-2)^1 x^1$, which is a
$\P$-position by Theorem \ref{thm:TwoHeapPposition}.
Suppose Player 1 applies the halving operation to $x^2 2^1$
to obtain $y^2 1^1$ where $3\le y \le 6$.
If $3\le y\le 5$, Player 2 removes 1 coin from the pile of height $1$
to obtain $y^2$,
which is a $\P$-position by Theorem \ref{thm:TwoHeapPposition}.
If $y=5$, Player 2 applies the halving operation to $6^2 1^1$
to obtain $3^2$, which is a 
$\P$-position by Theorem \ref{thm:TwoHeapPposition}.
Thus, $x^2 2^1$ is a $\P$-position.

Let $x\in\{14,22\}$.
We want to show that $x^1 (x+4)^1 2^1$
is a $\mathcal{P}$-position.
 After applying a standard Nim move on 
 $x^1 (x+4)^1 2^1$, the resulting position is
 either $x^1 y^1 2^1$, for
  $0\le y <x+4$, or $y^1 (x+4)^1 2^1$, for $0\le y <x$,
  or $x^1 (x+4)^1 y^1$, for $y\in\{0,1\}$.
   We observe that for all $y<x$,
  $\varphi_2(y)< x$.
  For the positions $w^1 y^1 2^1$, for
  $0\le y <x$ and $w\in\{x,x+4\}$,
  remove $w-\varphi_2(y)$ coins from the 
  heap of height $w$ to obtain 
  the position
  $y^1 (\varphi_2(y))^1 2^1$,
  which is a $\mathcal{P}$-position by the inductive hypothesis.
  For the position $x^1 (x+z)^1 2^1$ where
  $z\in\{01\}$, remove $2-z$ coins from the
  heap of height $z$ to obtain
the position
  $x^1 (x+z)^1 z^1$, which is a
  $\mathcal{P}$-position by Theorems
  \ref{thm:TwoHeapPposition} and
  \ref{thm:ThreeHeapab1Pposition}.
 For the position $x^1 y^1 2^1$ where
  $y\in\{x+2,x+3\}$, we apply the
  halving operation to obtain
the position
  $z^1 (z+1)^1 1^1$,
  for $z\in\{7,11\}$, which is a
  $\mathcal{P}$-position by Theorem
  \ref{thm:ThreeHeapab1Pposition}.
When we apply the halving operation on
$x^1 (x+4)^1 2^1$, the resulting position is
  $y^1 (y+2)^1 1^1$, for $y\in\{7,11\}$.
  We remove one coin from the heap of height
  $y+2$ to obtain 
  the position
  $y^1 (y+)^1 1^1$, for $y\in\{7,11\}$,  
  which is a
  $\mathcal{P}$-position by Theorem
  \ref{thm:ThreeHeapab1Pposition}.
Therefore, $x^1 (x+4)^1 2^1$ 
is a $\mathcal{P}$-position.
 
  Let $x\in\{15,23\}$.
We want to show that $x^1 (x+4)^1 2^1$
is a $\mathcal{P}$-position.
 After applying a standard Nim move on 
 $x^1 (x+4)^1 2^1$, the resulting position is
 either $x^1 y^1 2^1$, for
  $0\le y <x+4$, or $y^1 (x+4)^1 2^1$, for $0\le y <x$,
  or $x^1 (x+4)^1 y^1$, for $y\in\{0,1\}$.
   We observe that for all $y<x-1$,
  $\varphi_2(y)< x-1$.
  For the positions $w^1 y^1 2^1$, for
  $0\le y <x-1$ and $w\in\{x,x+4\}$,
  remove $w-\varphi_2(y)$ coins from the 
  heap of height $w$ to obtain 
  the position
  $y^1 (\varphi_2(y))^1 2^1$,
  which is a $\mathcal{P}$-position by the inductive hypothesis.
  When we remove one coin from the heap of height 
  either $x$ or $x+4$ from
  the position $x^1 (x+4)^`1 2^1$, 
  the resulting position is 
  either $(x-1)^1 (x+4)^1 2^1$ or
  $x^1 (x+3)^1 2^1$.
  We remove a coin from the opposing heap to obtain 
  the position $(x-1)^1 (x+3)^1 2^1$,
  which is a $\mathcal{P}$-position
  by the inductive hypothesis.
  For the positions $x^1 (x-y)^1 2^1$, 
  for $y\in\{0,1\}$,
  we remove the $2-y$ coins 
  from the heap of height 2 to obtain
  the position $x^1 (x-y)^1 y^1$, for $y\in\{0,1\}$,
  which is a 
  $\mathcal{P}$-position by Theorems
  \ref{thm:TwoHeapPposition} and
  \ref{thm:ThreeHeapab1Pposition}.
  For the positions $x^1 y^1 2^1$ for
  $y\in\{x+1,x+2\}$, we apply the halving
  operation to obtain the position
  $z^1 (z+1)^1 1^1$ for $z\in\{7,11\}$,
  which is a $\mathcal{P}$-position by Theorem
  \ref{thm:ThreeHeapab1Pposition}.
  For the position $x^1 (x+4)^1 y^1$, where
  $y\in\{01\}$, remove $4+y$ coins from the
  heap of height $x+4$ to obtain
the position
  $x^1 (x-y)^1 y^1$, for $y\in\{0,1\}$, which is a
  $\mathcal{P}$-position by Theorems
  \ref{thm:TwoHeapPposition} and
  \ref{thm:ThreeHeapab1Pposition}.
When we apply the halving operation on
$x^1 (x+4)^1 2^1$, the resulting position is
  $y^1 (y+2)^1 1^1$, for $y\in\{7,11\}$.
  We remove one coin from the heap of height
  $y+2$ to obtain 
  the position
  $y^1 (y+1)^1 1^1$, for $y\in\{7,11\}$,  
  which is a
  $\mathcal{P}$-position by Theorem
  \ref{thm:ThreeHeapab1Pposition}.
Therefore, $x^1 (x+4)^1 2^1$ 
is a $\mathcal{P}$-position.

  Let $x\in\{16,24\}$.
We want to show that $x^1 (x+4)^1 2^1$
is a $\mathcal{P}$-position.
 After applying a standard Nim move on 
 $x^1 (x+4)^1 2^1$, the resulting position is
 either $x^1 y^1 2^1$, for
  $0\le y <x+4$, or $y^1 (x+4)^1 2^1$, for $0\le y <x$,
  or $x^1 (x+4)^1 y^1$, for $y\in\{0,1\}$.
   We observe that for all $y<x-2$,
  $\varphi_2(y)< x-2$.
  For the positions $w^1 y^1 2^1$, for
  $0\le y <x-2$ and $w\in\{x,x+4\}$,
  remove $w-\varphi_2(y)$ coins from the 
  heap of height $w$ to obtain 
  the position
  $y^1 (\varphi_2(y))^1 2^1$,
  which is a $\mathcal{P}$-position by the inductive hypothesis.
  When we remove $y$ coins, for $y\in\{1,2\}$ 
  from the heap of height 
  either $x$ or $x+4$ from
  the position $x^1 (x+4)^1 2^1$, 
  the resulting position is 
  either $(x-y)^1 (x+4)^1 2^1$ or
  $x^1 (x+4-y)^1 2^1$.
  We remove $y$ coins from the 
  opposing heap to obtain 
  the position $(x-y)^1 (x+4-y)^1 2^1$,
  for $y\in\{1,2\}$,
  which is a $\mathcal{P}$-position
  by the inductive hypothesis.
  For the positions $x^1 (x+y)^1 2^1$, 
  for $y\in\{0,1\}$,
  we remove the $2-y$ coins 
  from the heap of height 2 to obtain
  the position $x^1 (x+y)^1 y^1$, 
  for $y\in\{0,1\}$, which is a 
  $\mathcal{P}$-position by Theorems
  \ref{thm:TwoHeapPposition} and
  \ref{thm:ThreeHeapab1Pposition}.
  For the position $x^1 y^1 2^1$ for
  $y\in\{x-2,x-1\}$, we apply the halving
  operation to obtain the position
  $z^1 (z+1)^1 1^1$ for $z\in\{8,12\}$,
  which is a $\mathcal{P}$-position by Theorem
  \ref{thm:ThreeHeapab1Pposition}.
  For the position $x^1 (x+4)^1 y^1$, where
  $y\in\{01\}$, remove $4-y$ coins from the
  heap of height $x+4$ to obtain
the position
  $x^1 (x+y)^1 y^1$, for $y\in\{0,1\}$, which is a
  $\mathcal{P}$-position by Theorems
  \ref{thm:TwoHeapPposition} and
  \ref{thm:ThreeHeapab1Pposition}.
When we apply the halving operation on
$x^1 (x+4)^1 2^1$, the resulting position is
  $y^1 (y+2)^1 1^1$, for $y\in\{8,12\}$.
  We remove 3 coins from the heap of height
  $y+2$ to obtain 
  the position
  $y^1 (y-1)^1 1^1$, for $y\in\{8,12\}$,  
  which is a
  $\mathcal{P}$-position by Theorem
  \ref{thm:ThreeHeapab1Pposition}.
Therefore, $x^1 (x+4)^1 2^1$ 
is a $\mathcal{P}$-position.

  Let $x\in\{17,25\}$.
We want to show that $x^1 (x+4)^1 2^1$
is a $\mathcal{P}$-position.
 After applying a standard Nim move on 
 $x^1 (x+4)^1 2^1$, the resulting position is
 either $x^1 y^1 2^1$, for
  $0\le y <x+4$, or $y^1 (x+4)^1 2^1$, for $0\le y <x$,
  or $x^1 (x+4)^1 y^1$, for $y\in\{0,1\}$.
   We observe that for all $y<x-3$,
  $\varphi_2(y)< x-3$.
  For the positions $w^1 y^1 2^1$, for
  $0\le y <x-3$ and $w\in\{x,x+4\}$,
  remove $w-\varphi_2(y)$ coins from the 
  heap of height $w$ to obtain 
  the position
  $y^1 (\varphi_2(y))^1 2^1$,
  which is a $\mathcal{P}$-position by the inductive hypothesis.
  When we remove $y$ coins, for $y\in\{1,2,3\}$ 
  from the heap of height 
  either $x$ or $x+4$ from
  the position $x^1 (x+4)^1 2^1$, 
  the resulting position is 
  either $(x-y)^1 (x+4)^1 2^1$ or
  $x^1 (x+4-y)^1 2^1$.
  We remove $y$ coins from the 
  opposing heap to obtain 
  the position $(x-y)^1 (x+4-y)^1 2^1$,
  for $y\in\{1,2,3\}$,
  which is a $\mathcal{P}$-position
  by the inductive hypothesis.
  For the positions $x^1 (x-y)^1 2^1$, 
  for $y\in\{0,1\}$,
  we remove the $2-y$ coins 
  from the heap of height 2 to obtain
  the position $x^1 (x-y)^1 y^1$, 
  for $y\in\{0,1\}$, which is a 
  $\mathcal{P}$-position by Theorems
  \ref{thm:TwoHeapPposition} and
  \ref{thm:ThreeHeapab1Pposition}.
  For the position $x^1 y^1 2^1$ for
  $y\in\{x-3,x-2\}$, we apply the halving
  operation to obtain the position
  $z^1 (z-1)^1 1^1$ for $z\in\{8,12\}$,
  which is a $\mathcal{P}$-position by Theorem
  \ref{thm:ThreeHeapab1Pposition}.
  For the position $x^1 (x+4)^1 y^1$, where
  $y\in\{01\}$, remove $4+y$ coins from the
  heap of height $x+4$ to obtain
the position
  $x^1 (x-y)^1 y^1$, for $y\in\{0,1\}$, which is a
  $\mathcal{P}$-position by Theorems
  \ref{thm:TwoHeapPposition} and
  \ref{thm:ThreeHeapab1Pposition}.
When we apply the halving operation on
$x^1 (x+4)^1 2^1$, the resulting position is
  $y^1 (y+2)^1 1^1$, for $y\in\{8,12\}$.
  We remove 3 coins from the heap of height
  $y+2$ to obtain 
  the position
  $y^1 (y-1)^1 1^1$, for $y\in\{8,12\}$,  
  which is a
  $\mathcal{P}$-position by Theorem
  \ref{thm:ThreeHeapab1Pposition}.
Therefore, $x^1 (x+4)^1 2^1$ 
is a $\mathcal{P}$-position.

Suppose either $30\le x<56$ or $7\cdot 2^{2n}\le x< 7\cdot 2^{2n+1}$ for some 
integer $n\ge 2$ and,
for all nonnegative integers $y<x$,
  $y^1 z^1 2^1$ is a $\mathcal{P}$-position if and 
  only if $z=\varphi_2(y)$.
We need to establish that $x^2 2^1$ is 
a $\mathcal{P}$-position.
 We observe that $\varphi_2(y)<x$ for all $y <x$.
Suppose Player 1 applies a standard Nim move to $x^2 2^1$.
By Remark \ref{rem:BeginInduction},
we may assume that the resulting position is 
   $x^2 y^1$ for $y\in\{0,1\}$.

Suppose, for $y\in\{1,\, 2\}$, Player 1 removes $y$ coins from the 
pile of height 2 to obtain $x^2 (2-y)^1$.
   The position  $x^2 (2-y)^1$
   is an $\mathcal{N}$-position by Theorems
   \ref{thm:TwoHeapPposition} 
   and \ref{thm:ThreeHeapab1Pposition}.
  Suppose Player 1 applies the halving operation 
  to $x^2 2^1$ to obtain $y^2 1^1$ where $y=\lfloor \tfrac{x}{2}\rfloor$.
  Since either $15\le y<28$ or $7\cdot 2^{2n-1}\le y<7\cdot 2^{2n}$,
  $\varphi_1(y)=y\oplus 1\ne y$ by Theorem \ref{thm:ThreeHeapab1Pposition}.
  Thus, $y^2 1^1$ is an $\mathcal{N}$-position.
  Therefore, $x^2 2^1$ is a $\mathcal{P}$-position.

Suppose $7\cdot 2^{2n-1}\le 4x< 7\cdot 2^{2n}$ for some 
integer $n\ge 2$ and,
for all nonnegative integers $y<4x$,
  $y^1 z^1 2^1$ is a $\mathcal{P}$-position if and 
  only if $z=\varphi_2(y)$.
We want to show that $(4x+r)^1 (4x+2+r)^1 2^1$
is a $\mathcal{P}$-position for $r\in\{0,\, 1\}$.
We observe that for all $y<4x$, $\varphi_2(y)<4x$.
Suppose Player 1 applies a standard Nim move to $(4x+r)^1 (4x+2+r)^1 2^1$.
By Remark \ref{rem:BeginInduction}, we may assume that
the resulting position is
 either $y^1 (4x+2+r)^1 2^1$ for $4x\le y <4x+r$,
 $(4x+r)^1 y^1 2^1$ for $4x\le y <4x+2+r$, 
  or $(4x+r)^1 (4x+2+r)^1 y^1$ for $0\le y <2$.

  Assume $r=0$. 
  Suppose, for $y\in\{1,\, 2\}$, Player 1 removes $y$ coins from the 
  pile of height either $4x+2$ or $2$ to obtain 
  $(4x)^1 (4x+2-y)^1 2^1$ or $(4x)^1 (4x+2)^1 (2-y)^1$.
  Player 2 removes $y$ coins from the opposing pile 
  to obtain the position
$(4x)^1 (4x+2-y)^1 (2-y)^1$,
  which is a $\mathcal{P}$-position 
  by Theorems \ref{thm:TwoHeapPposition} and
  \ref{thm:ThreeHeapab1Pposition}.
  
  Assume $r\in\{0,\, 1\}$.
  Suppose Player 1 applies the halving operation to 
  $(4x+r)^1 (4x+2+r)^1 2^1$ to obtain
  $(2x)^1 (2x+1)^1 1^1$ 
  where $7\cdot 2^{2n-2} \le 2x < 7\cdot 2^{2n-1}$.
  By Lemma~\ref{lem:HalveCase0}, $(2x)^1 (2x+1)^1 1^1$ 
  is an $\N$-position.
  Hence, $(4x)^1 (4x+2)^1 2^1$ is a $\mathcal{P}$-position.

  Assume $r=1$.
Suppose Player 1 removes 1 coin from the pile of height 
either $4x+1$ or $4x+3$
to obtain $(4x)^1 (4x+3)^1 2^1$ or $(4x+1)^1 (4x+2)^1 2^1$.
Player 2 removes 1 coin from the opposing pile to
obtain $(4x)^1 (4x+2)^1 2^1$, which is a 
$\mathcal{P}$-position.
Suppose, for $y\in\{2,\, 3\}$, Player 1 removes $y$ coins 
from the pile of height $4x+3$
to obtain $(4x+1)^1 (4x+3-y)^1 2^1$.
Player 2 removes $4-y$ coins from the pile of height 2 to
obtain $(4x+1)^1 (4x+3-y)^1 (y-2)^1$, which is a 
$\mathcal{P}$-position by Theorems \ref{thm:TwoHeapPposition} and \ref{thm:ThreeHeapab1Pposition}.
Suppose, for $y\in\{1,\, 2\}$, Player 1 removes $y$ coins 
from the pile of height $2$
to obtain $(4x+1)^1 (4x+3)^1 (2-y)^1$.
Player 2 removes $4-y$ coins from the pile of height $4x+3$ to
obtain $(4x+1)^1 (4x+y-1)^1 (2-y)^1$, which is a 
$\mathcal{P}$-position by Theorems \ref{thm:TwoHeapPposition} and \ref{thm:ThreeHeapab1Pposition}.
  Therefore, $(4x+1)^1 (4x+3)^1 2^1$ is a $\mathcal{P}$-position.
  
Suppose $a,b\ge 30$.
Let $m$ be the positive integer such that $7\cdot 2^m\le a<7\cdot 2^{m+1}$.
If $m$ is odd, $\gamma_2(a)=2$.
If $m$ is even, $\gamma_2(a)=0$.
Thus, $a^1 b^1 2^1$ is a $\P$-position if and only if $a\oplus b=\gamma_2(a)$.\qed
\end{proof}

\begin{table}[t]
\centering
\begin{tabular}{ccccccccccccccc}
\multicolumn{10}{c}{List of games $a^1 b^1 3^1$ that are
$\mathcal{P}$-positions} \\
\hline
(0,3) &(1,1)  &(2,4)  &(3,0)  &(4,2)   &(5,7)
&(6,6)  &(7,5)  &(8,9)  &(9,8)  
\\[0.025in]
(10,11) &(11,10)  &(12,13)  &(13,12)  &(14,19)  &(15,18)
&(16,21)  &(17,20)  &(18,15)  &(19,14)  
\\[0.025in]
(20,17) &(21,16)  &(22,27)  &(23,26)  &(24,29)   &(25,28)
&(26,23)  &(27,22)  &(28,25)  &(29,24) 
\\[0.025in]
(30,31) &(31,30)  &(32,33)  &(33,32)  &(34,35)   &(35,34)
&(36,37)  &(37,36)  &(38,39)  &(39,38) 
\\[0.025in]
(40,41) &(41,40)  &(42,43)  &(43,42)  &(44,45)   &(45,44)
&(46,47)  &(47,46)  &(48,49)  &(49,48)  
\\[0.025in]
(50,51) &(51,50)  &(52,53)  &(53,52)  &(54,55)   &(55,54)
&(56,59)  &(57,58)  &(58,57)  &(59,56) 
\\[0.025in]
(60,63) &(61,62)  &(62,61)  &(63,60)  &(64,67)   &(65,66)
&(66,65)  &(67,64)  &(68,71)  &(69,70)  
\\[0.025in]
(70,69) &(71,68)  &(72,75)  &(73,74)  &(74,73)   &(75,72)
&(76,79)  &(77,78)  &(78,77)  &(79,76) 
\\[0.025in]
\end{tabular} 
\caption{$\mathcal{P}$-position games $a^1 b^1 3^1$ 
listed as $(a,b)$.}
\label{tab:ab3Ppoisitions}
\end{table}

The $\mathcal{P}$-positions of the three-pile
games $a^1 \, b^1 \, 3^1$, for $0\le a \le 79$,
are listed in Table \ref{tab:ab3Ppoisitions}
as the ordered pair $(a,b)$.
We characterize the three-pile games $a^1 b^1 3^1$ that
are $\mathcal{P}$-positions in Theorem \ref{thm:ThreeHeapab3Pposition}.

\begin{theorem}\label{thm:ThreeHeapab3Pposition}
The first 30 values of $\varphi_3$, in the form of 
$(a,\varphi_3(a))$, are displayed in Table \ref{tab:ab3Ppoisitions}.
   Let $x$ be a positive integer. 
   \begin{enumerate}
  \item Suppose either $30\le x<56$ or $7\cdot 2^{2n}\le x < 7\cdot 2^{2n+1}$ for some integer $n\ge 2$.
       Then 
       \begin{equation*}
         \varphi_3(x)=x \oplus 1.  
       \end{equation*}
       \item Suppose $7\cdot 2^{2n-1} \le x < 7\cdot 2^{2n}$ 
       for some integer $n\ge 2$.
       Then 
       \begin{equation*}
        \varphi_3(x)= x \oplus 3.   
       \end{equation*}
   \end{enumerate}
   Hence, $a^1 b^1 3^1$ is a $\mathcal{P}$-position if and only if
   $b=\varphi_3(a)$.
   Furthermore, for $a,b\ge 30$,
   $a^1 b^1 3^1$ is a $\P$-position if and only if
    \begin{equation}\label{eqn:3PileResult}
       a\oplus b= 2 + (-1)^{\lfloor \log_2(a/7)\rfloor+1}.
    \end{equation}
\end{theorem}

\begin{remark}\label{rem:PhiFunction}
    We have $\varphi_c(x)=\varphi_0(x)\oplus c$
    for $c\in\{1,2,3\}$ and $x\ge 30$.
\end{remark}

\begin{proof}
It is straight forward to verify that, for 
$0\le x < 30$, $x^1 y^1 3^1$ 
is a $\mathcal{P}$-position if and only if
$y=\varphi_3(x)$. We leave the details of the 
proof to the reader.
This establishes the base case.

Suppose either $30\le 2x<56$ or $7\cdot 2^{2n}\le 2x< 7\cdot 2^{2n+1}$ for some 
integer $n\ge 2$ and, 
for all nonnegative integers $y<2x$,
  $y^1 z^1 3^1$ is a $\mathcal{P}$-position if and 
  only if $z=\varphi_3(y)$.
We need to establish that $(2x)^1 (2x+1)^1 3^1$ is 
a $\mathcal{P}$-position.
 We observe that for all $y <2x$, $\varphi_3(y)<2x$.
Suppose Player 1 applies a standard Nim move to $(2x)^1 (2x+1)^1 3^1$.
By Remark \ref{rem:BeginInduction}, we may assume that
  the resulting position is either
    $(2x)^1 y^1 3^1$ for $2x\le y <2x+1$ or 
  $(2x)^1 (2x+1)^1 y^1$ for $0\le y <3$.
  
Suppose Player 1 removes 1 coin from the pile of height either $2x+1$ or $3$
to obtain $(2x)^2 3^1$ or $(2x)^1 (2x+1)^1 2^1$.
Player 2 removes 1 coin from the opposing pile to obtain $(2x)^2 2^1$,
which is a $\mathcal{P}$-position by Theorem \ref{thm:ThreeHeapab2Pposition}.
Suppose, for $y\in\{2,\, 3\}$, Player 1 removes $y$ coins from the pile of height 3
to obtain $(2x)^1 (2x+1)^1 (3-y)^1$.
Player 2 applies the halving operation 
to $(2x)^1 (2x+1)^1 (3-y)^1$ to obtain $x^2$,
which is a $\mathcal{P}$-position by Theorem \ref{thm:TwoHeapPposition}.
Suppose Player 1 applies the halving operation to $(2x)^1 (2x+1)^1 3^1$
to obtain $x^2 1^1$.
Player 2 removes 1 coin from the pile of height 1 to obtain $x^2$,
which is a $\mathcal{P}$-position by Theorem \ref{thm:TwoHeapPposition}.
Hence, $(2x)^1 (2x+1)^1 3^1$ is a $\mathcal{P}$-position.

 Suppose $7\cdot 2^{2n-1}\le 4x< 7\cdot 2^{2n}$ for some 
integer $n\ge 2$ and,  
for all nonnegative integers $y<4x$,
  $y^1 z^1 3^1$ is a $\mathcal{P}$-position if and 
  only if $z=\varphi_3(y)$.
We want to show that $(4x+r)^1 (4x+3-r)^1 3^1$
is a $\mathcal{P}$-position for $r\in\{0,\, 1\}$.
We observe that for all $y<4x$,
  $\varphi_3(y)<4x$.
  Suppose Player 1 applies a standard Nim move to $(4x+r)^1 (4x+3-r)^1 3^1$.
By Remark \ref{rem:BeginInduction}, we may assume that
  the resulting position is either
 $y^1 (4x+3-r)^1 3^1$ for $4x\le y <4x+r$,
or $(4x+r)^1 y^1 3^1$ for $4x\le y <4x+3-r$, 
  or $(4x+r)^1 (4x+3-r)^1 y^1$ for $0\le y <3$.
  
  Assume $r=0$.
  Suppose, for $y\in\{1,\, 2,\, 3\}$, Player 1 removes $y$ coins from
  the pile of height either $4x+3$ or $3$ to obtain
  either $(4x)^1 (4x+3-y)^1 3^1$ or $(4x)^1 (4x+3)^1 (3-y)^1$.
  Player 2 removes $y$ coins from the opposing pile to 
  obtain $(4x)^1 (4x+3-y)^1 (3-y)^1$, which is a $\mathcal{P}$-position by 
  Theorems \ref{thm:TwoHeapPposition}, 
  \ref{thm:ThreeHeapab1Pposition}, and 
  \ref{thm:ThreeHeapab2Pposition}.

   Assume $r=\{0,\, 1\}$.
  Suppose Player 1 applies the halving operation 
  to $(4x+r)^1 (4x+3-r)^1 3^1$
  to obtain $(2x)^1 (2x+1)^1 1^1$ 
  where $7\cdot 2^{2n-2} \le 2x < 7\cdot 2^{2n-1}$.
   By Lemma~\ref{lem:HalveCase0}, $(2x)^1 (2x+1)^1 1^1$ 
  is an $\N$-position.
  Hence, $(4x)^1 (4x+3)^1 3^1$ is a $\mathcal{P}$-position.

  Assume $r=1$. Suppose Player 1 removes 1 coin from the pile
  of height either $4x+1$ or $3$ to obtain $(4x)^1 (4x+2)^1 3^1$
  or $(4x+1)^1 (4x+2)^1 2^1$.
  Player 2 removes 1 coin from the opposing pile to obtain
  $(4x)^1 (4x+2)^1 2^1$, which is a $\mathcal{P}$-position by 
  Theorem \ref{thm:ThreeHeapab2Pposition}.
 Suppose, for $y\in\{1,\, 2\}$, Player 1 removes $y$ coins from the 
 pile of height $4x+2$ to obtain $(4x+1)^1 (4x+2-y)^1 3^1$.
 Player 2 removes $4-y$ coins from the pile of height $3$
 to obtain  $(4x+1)^1 (4x+2-y)^1 (y-1)^1$, which is a
 $\mathcal{P}$-position by Theorems \ref{thm:TwoHeapPposition} 
 and \ref{thm:ThreeHeapab1Pposition}.
 Suppose, for $y\in\{2,\, 3\}$, Player 1 removes $y$ coins from the 
 pile of height $3$ to obtain $(4x+1)^1 (4x+2)^1 (3-y)^1$.
 Player 2 removes $4-y$ coins from the pile of height $4x+2$
 to obtain  $(4x+1)^1 (4x+y-2)^1 (3-y)^1$, which is a
 $\mathcal{P}$-position by Theorems \ref{thm:TwoHeapPposition} 
 and \ref{thm:ThreeHeapab1Pposition}.
Therefore, $(4x+1)^1 (4x+2)^1 3^1$ is a $\mathcal{P}$-position.

Suppose $a,b\ge 30$.
Let $m$ be the positive integer such that $7\cdot 2^m\le a<7\cdot 2^{m+1}$.
If $m$ is odd, $\gamma_3(a)=3$.
If $m$ is even, $\gamma_3(a)=1$.
Thus, $a^1 b^1 3^1$ is a $\P$-position if and only if $a\oplus b=\gamma_3(a)$.\qed
\end{proof}

\begin{table}[t]
\centering
\begin{tabular}{ccccccccccccccc}
\multicolumn{10}{c}{List of games $a^1 b^1 4^1$ that are
$\mathcal{P}$-positions} \\
\hline
(0,4) &(1,5)  &(2,3)  &(3,2)  &(4,0)   &(5,1)
&(6,7)  &(7,6)  &(8,10)  &(9,11)  
\\[0.025in]
(10,8) &(11,9)  &(12,14)  &(13,15)  &(14,12)  &(15,13)
&(16,18)  &(17,19)  &(18,16)  &(19,17)  
\\[0.025in]
(20,22) &(21,23)  &(22,20)  &(23,21)  &(24,26)   &(25,27)
&(26,24)  &(27,25)  &(28,28)  &(29,29)  
\\[0.025in]
(30,32) &(31,33)  &(32,30)  &(33,31)  &(34,36)   &(35,37)
&(36,34)  &(37,35)  &(38,40)  &(39,41)  
\\[0.025in]
(40,38) &(41,39)  &(42,44)  &(43,45)  &(44,42)   &(45,43)
&(46,48)  &(47,49)  &(48,46)  &(49,47)  
\\[0.025in]
(50,52) &(51,53)  &(52,50)  &(53,51)  &(54,56)   &(55,57)
&(56,54)  &(57,55)  &(58,60)  &(59,61)  
\\[0.025in]
(60,58) &(61,59)  &(62,64)  &(63,65)  &(64,62)   &(65,63)
&(66,68)  &(67,69)  &(68,66)  &(69,67)  
\\[0.025in]
(70,72) &(71,73)  &(72,70)  &(73,71)  &(74,76)   &(75,77)
&(76,74)  &(77,75)  &(78,80)  &(79,81)  
\\[0.025in]
\end{tabular} 
\caption{$\mathcal{P}$-position games $a^1 b^1 4^1$ 
listed as $(a,b)$.}
\label{tab:ab4Ppoisitions}
\end{table}

In the next two lemmas,
we use Theorems \ref{thm:TwoHeapPposition},
\ref{thm:ThreeHeapab1Pposition}, \ref{thm:ThreeHeapab2Pposition},
and \ref{thm:ThreeHeapab3Pposition}
to determine when some 
three-pile games are $\mathcal{N}$-positions.

\begin{lemma}\label{lem:NPositions}
Let $a$, $b$, and $c$ be integers such that $a,b\ge 30$ and $c\ge 4$.
If $a\oplus b\in\{0,\, 1,\, 2,\, 3\}$, then the three-pile game
$a^1 b^1 c^1$ is an $\mathcal{N}$-position.
\end{lemma}

\begin{proof}
By Theorem \ref{thm:TwoHeapPposition},
we have $\varphi_0(a)=a$ when $7\cdot 2^{2n-1}\le a<7\cdot 2^{2n}$
for some positive integer $n$ and
$\varphi_0(a)=a\oplus 2$ when $7\cdot 2^{2n}\le a<7\cdot 2^{2n+1}$
for some positive integer $n$.
By Remark \ref{rem:PhiFunction},
$\varphi_k(a)=\varphi_0(a)\oplus k$ for $k\in\{0,\, 1,\, 2,\, 3\}$.
Let $\ell=a\oplus b$. 
Then there exists an integer $k\in\{0,\, 1,\, 2,\, 3\}$
such that $\varphi_k(a)=a\oplus \ell$.
Given the game $a^1 b^1 c^1$,
Player 1 removes $c-k$ coins from the pile of height $c$
to obtain $a^1 b^1 k^1$, which is a $\P$-position
by either Theorem \ref{thm:TwoHeapPposition},
\ref{thm:ThreeHeapab1Pposition}, \ref{thm:ThreeHeapab2Pposition},
or \ref{thm:ThreeHeapab3Pposition}
because $\varphi_k(a)=a\oplus \ell=b$.
Therefore, $a^1 b^1 c^1$ is an $\N$-position.
\end{proof}

\begin{lemma}\label{lem:NPositions2}
Let $a$, $b$, and $c$ be integers such that $a,b\ge 30$ and $c\le 3$.
If $a\oplus b\ge 4$, then the three-pile game
$a^1 b^1 c^1$ is an $\mathcal{N}$-position.
\end{lemma}

\begin{proof}
    By Theorems \ref{thm:TwoHeapPposition},
    \ref{thm:ThreeHeapab1Pposition}, \ref{thm:ThreeHeapab2Pposition},
    and \ref{thm:ThreeHeapab3Pposition},
    if $a^1 b^1 c^1$ is a $\P$-position,
    then $a\oplus b\in\{0,\, 1,\, 2,\, 3\}$.
    Since $a\oplus b\ge 4$, $a^1 b^1 c^1$ is an $\N$-position.
\end{proof}

We introduce a variation of bitwise addition that we will find useful in our next results
on three-pile games of Halve Nim.

\begin{definition}\label{defn:ShiftedBitwiseSum}
    Let $n$ be a fixed positive integer. 
    For nonnegative integers $x$ and $y$, the 
    \emph{$n$-shifted bitwise sum} of $x$ and $y$ is given by
    \begin{equation*}
        x \oplus_n y = \big((x+n) \oplus y\big)  -n.
\end{equation*}
\end{definition}

\begin{remark}\label{rem:ShifetedBitWiseSum}
    The $n$-shifted bitwise sum is not commutative.
    Furthermore, $ x=y \oplus_n  k$
    is equivalent to 
    \begin{equation*}
        (x+n) \oplus (y+n) = k.
\end{equation*}
\end{remark}

The $\mathcal{P}$-positions of the three-pile
games $a^1 \, b^1 \, 4^1$, for $0\le a \le 79$,
are listed in Table \ref{tab:ab4Ppoisitions}
as the ordered pair $(a,b)$.
We characterize the three-pile games $a^1 b^1 4^1$ that
are $\mathcal{P}$-positions in Theorem \ref{thm:ThreeHeapab4Pposition}.
In the next two theorems, we need to verify that the game
$(2x+1)^1 (2x+2)^1 2^1$ is an $\N$-position.

\begin{lemma}\label{lem:HalveCase}
    Suppose $x\ge 7$. Then the game $(2x+1)^1 (2x+2)^1 2^1$ is an $\N$-position.
\end{lemma}

\begin{proof}
    First, assume $x\ge 15$.
 Since $2x+1$ and $2x+2$ have opposite parity,
 $(2x+1)\oplus (2x+2)$ is odd.
 Thus, $(2x+1)\oplus (2x+2)\not\in\{0,\, 2\}$.
 By Theorem \ref{thm:ThreeHeapab2Pposition},
 $(2x+1)^1 (2x+2)^1 2^1$ is an $\mathcal{N}$-position.

Second, assume $7\le x \le 14$.
If $x\in\{7,\,11\}$, Player 1 applies the halving operation to
 $(2x+1)^1 (2x+2)^1 2^1$ to obtain  $x^1 (x+1)^1 1^1$.
 If $x\in\{8,\, 12\}$, Player 1 removes 3 coins from the pile 
 of height $2x+1$ to obtain  $(2x-2)^1 (2x+2)^1 2^1$.
 If $x\in\{9,\, 10,\, 13,\, 14\}$, Player 1 removes 5 coins 
 from the pile of height $2x+2$ to obtain $(2x+1)^1 (2x-3)^1 2^1$.
 One observes from Tables \ref{tab:ab1Ppoisitions} and
 \ref{tab:ab2Ppoisitions} that each of the Player 1 moves results 
 in a $\mathcal{P}$-position.
\end{proof}

\begin{theorem}\label{thm:ThreeHeapab4Pposition}
The first 30 values of $\varphi_4$, in the form of 
$(a,\varphi_4(a))$, are displayed in Table \ref{tab:ab4Ppoisitions}.
Suppose $x\ge 30$. 
       Then 
       \begin{equation*}
        \varphi_4(x)= x \oplus_2 2.   
       \end{equation*}
   Hence, $a^1 b^1 4^1$ is a $\mathcal{P}$-position if and only if
   $b=\varphi_4(a)$.
   Furthermore, for $a,b\ge 30$,
   $a^1 b^1 4^1$ is a $\mathcal{P}$-position if and only if
   $(a+2)\oplus (b+2)=2$.
\end{theorem}

\begin{proof}
It is straight forward to verify that, for 
$0\le x < 30$, $x^1 y^1 4^1$ 
is a $\mathcal{P}$-position if and only if
$y=\varphi_4(x)$. We leave the details of the 
proof to the reader.
This establishes the base case.

We need to show that $(4x+2+r)^1 (4x+4+r)^1 4^1$
is a $\mathcal{P}$-position for $x\ge7$ and $r\in\{0,\, 1\}$.
We assume that for $y<4x+2$, $y^1 z^1 4^1$ is a 
$\mathcal{P}$-position if and only if $z=\varphi_4(y)$.
We observe that for all $y<4x+2$,
  $\varphi_4(y)<4x+2$.
Suppose Player 1 applies a standard Nim move to $(4x+2+r)^1 (4x+4+r)^1 4^1$.
By Remark \ref{rem:BeginInduction}, we may assume that 
the resulting position is
 either $y^1 (4x+4+r)^1 4^1$ for $4x+2\le y <4x+2+r$, 
 $(4x+2+r)^1 y^1 4^1$ for $4x+2\le y <4x+4+r$, 
  or $(4x+2+r)^1 (4x+4+r)^1 y^1$ for $0\le y <4$.

  Suppose, for $y\in\{1,\, 2\}$, Player 1 removes $y+r$ coins from
  the pile of height $4x+4+r$ 
  to obtain $(4x+2+r)^1 (4x+4-y)^1 4^1$.
  We observe that $(4x+2+r)\oplus (4x+4-y)\in\{0,\, 1\}$.
  By Lemma \ref{lem:NPositions},  $(4x+2+r)^1 (4x+4-y)^1 4^1$
  is an $\mathcal{N}$-position.

  Suppose, for $1\le y\le 4$, Player 1 removes $y$ coins from
  the pile of height $4$ 
  to obtain $(4x+2+r)^1 (4x+4+r)^1 (4-y)^1$.
  Since $4x+2+r<4(x+1)\le 4x+4+r$,
  the binary representations of  $4x+2+r$ and  $4x+4+r$
  differ in the $2^2$ place value.
  Thus, $(4x+2+r)\oplus (4x+4+r)\ge 4$.
  By Lemma \ref{lem:NPositions2},  $(4x+2+r)^1 (4x+4+r)^1 (4-y)^1$
  is an $\mathcal{N}$-position.
  
 Suppose Player 1 applies the halving operation 
  to $(4x+2+r)^1 (4x+4+r)^1 4^1$
  to obtain $(2x+1)^1 (2x+2)^1 2^1$.
  By Lemma \ref{lem:HalveCase},
 $(2x+1)^1 (2x+2)^1 2^1$ is an $\mathcal{N}$-position.
 Hence, $(4x+2)^1 (4x+4)^1 4^1$ is a $\mathcal{P}$-position.

 Assume $r=1$.
 Suppose Player 1 removes 1 coin from the pile of
 height either $4x+3$ or $4x+5$ to obtain 
 $(4x+2)^1 (4x+5)^1 4^1$ or $(4x+3)^1 (4x+4)^1 4^1$.
 Player 2 removes 1 coin from the opposing pile to obtain
 $(4x+2)^1 (4x+4)^1 4^1$, which is a $\mathcal{P}$-position.
 Therefore, $(4x+3)^1 (4x+5)^1 4^1$ is a $\mathcal{P}$-position.
 
 By Remark \ref{rem:ShifetedBitWiseSum}, $b=\varphi_4(a)=a\oplus_2 2$
 is equivalent to $(a+2)\oplus (b+2)=2$.\qed
\end{proof}

\begin{table}[t]
\centering
\begin{tabular}{ccccccccccccccc}
\multicolumn{10}{c}{List of games $a^1 b^1 5^1$ that are
$\mathcal{P}$-positions} \\
\hline
(0,5) &(1,4)  &(2,6)  &(3,7)  &(4,1)   &(5,0)
&(6,2)  &(7,3)  &(8,11)  &(9,10)  
\\[0.025in]
(10,9) &(11,8)  &(12,15)  &(13,14)  &(14,13)  &(15,12)
&(16,19)  &(17,18)  &(18,17)  &(19,16)  
\\[0.025in]
(20,23) &(21,22)  &(22,21)  &(23,20)  &(24,27)   &(25,26)
&(26,25)  &(27,24)  &(28,29)  &(29,28)  
\\[0.025in]
(30,33) &(31,32)  &(32,31)  &(33,30)  &(34,37)   &(35,36)
&(36,35)  &(37,34)  &(38,41)  &(39,40)  
\\[0.025in]
(40,39) &(41,38)  &(42,45)  &(43,44)  &(44,43)   &(45,42)
&(46,49)  &(47,48)  &(48,47)  &(49,46)  
\\[0.025in]
(50,53) &(51,52)  &(52,51)  &(53,50)  &(54,57)   &(55,56)
&(56,55)  &(57,54)  &(58,61)  &(59,60)  
\\[0.025in]
(60,59) &(61,58)  &(62,65)  &(63,64)  &(64,63)   &(65,62)
&(66,69)  &(67,68)  &(68,67)  &(69,66)  
\\[0.025in]
(70,73) &(71,72)  &(72,71)  &(73,70)  &(74,77)   &(75,76)
&(76,75)  &(77,74)  &(78,81)  &(79,80)  
\\[0.025in]
\end{tabular} 
\caption{$\mathcal{P}$-position games $a^1 b^1 5^1$ 
listed as $(a,b)$.}
\label{tab:ab5Ppoisitions}
\end{table}

The $\mathcal{P}$-positions of the three-pile
games $a^1 \, b^1 \, 5^1$, for $0\le a \le 79$,
are listed in Table \ref{tab:ab5Ppoisitions}
as the ordered pair $(a,b)$.
We characterize the three-pile games $a^1 b^1 5^1$ that
are $\mathcal{P}$-positions in Theorem \ref{thm:ThreeHeapab5Pposition}.

\begin{theorem}\label{thm:ThreeHeapab5Pposition}
The first 30 values of $\varphi_5$, in the form of 
$(a,\varphi_5(a))$, are displayed in Table \ref{tab:ab5Ppoisitions}.
Suppose $x\ge 30$. 
       Then 
       \begin{equation*}
        \varphi_5(x)= x \oplus_2 3.   
       \end{equation*}
   Hence, $a^1 b^1 5^1$ is a $\mathcal{P}$-position if and only if
   $b=\varphi_5(a)$.
   Furthermore, for $a,b\ge 30$,
   $a^1 b^1 5^1$ is a $\mathcal{P}$-position if and only if
   $(a+2)\oplus (b+2)=3$.
\end{theorem}

\begin{proof}
It is straight forward to verify that, for 
$0\le x < 30$, $x^1 y^1 5^1$ 
is a $\mathcal{P}$-position if and only if
$y=\varphi_5(x)$. We leave the details of the 
proof to the reader.
This establishes the base case.

We want to show that $(4x+2+r)^1 (4x+5-r)^1 5^1$ is a
    $\mathcal{P}$-position for $x\ge 7$ and $r\in\{0,\, 1\}$.
 We assume that for $y<4x+2$, $y^1 z^1 5^1$ is a 
$\mathcal{P}$-position if and only if $z=\varphi_5(y)$.
We observe that for all $y<4x+2$,
  $\varphi_5(y)<4x+2$.
  Suppose Player 1 applies a standard Nim move to  $(4x+2+r)^1 (4x+5-r)^1 5^1$.
  By Remark \ref{rem:BeginInduction}, we may assume that
    the resulting position is
 either $y^1 (4x+5-r)^1 5^1$ for $4x+2\le y <4x+2+r$,
 $(4x+2+r)^1 y^1 5^1$ for $4x+2\le y <4x+5-r$, 
  or $(4x+2+r)^1 (4x+5-r)^1 y^1$ for $0\le y<5$.

Assume $r=0$. Suppose Player 1 removes 1 coin from the pile of
height either $4x+5$ or $5$ to obtain  $(4x+2)^1 (4x+4)^1 5^1$ or $(4x+2)^1 (4x+5)^1 4^1$.
Player 2 removes 1 coin from the opposing pile to
obtain $(4x+2)^1 (4x+4)^1 4^1$, which is a 
$\mathcal{P}$-position by Theorem 
\ref{thm:ThreeHeapab4Pposition}.

Assume $r\in\{0,\, 1\}$.
Suppose, for $y\in\{2,\, 3\}$, Player 1 removes $y-r$ coins from
the pile of height 
$4x+5-r$ to obtain
$(4x+2+r)^1 (4x+5-y)^1 5^1$.
Since $(4x+2+r)\oplus (4x+5-y)\in\{0,\, 1\}$,
$(4x+2+r)^1 (4x+5-y)^1 5^1$ is an $\mathcal{N}$-position 
by Lemma \ref{lem:NPositions}.

 Suppose, for $2\le y\le 5$, Player 1 removes $y$ coins from
  the pile of height $5$ 
  to obtain $(4x+2+r)^1 (4x+5-r)^1 (5-y)^1$.
  Since $4x+2+r<4(x+1)\le 4x+5-r$,
  the binary representations of  $4x+2+r$ and  $4x+5-r$
  differ in the $2^2$ place value.
  Thus, $(4x+2+r)\oplus (4x+5-r)\ge 4$.
  By Lemma \ref{lem:NPositions2},  $(4x+2+r)^1 (4x+5-r)^1 (5-y)^1$
  is an $\mathcal{N}$-position.

  Suppose Player 1 applies the halving operation 
  to $(4x+2+r)^1 (4x+5-r)^1 5^1$
  to obtain $(2x+1)^1 (2x+2)^1 2^1$.
  By Lemma \ref{lem:HalveCase},
 $(2x+1)^1 (2x+2)^1 2^1$ is an $\mathcal{N}$-position.
  Hence, $(4x+2)^1 (4x+5)^1 5^1$ is a 
  $\mathcal{P}$-position.

  Assume $r=1$. 
  Suppose Player 1 removes 1 coin from the pile
  of height $4x+3$ or $5$ to obtain $(4x+2)^1 (4x+4)^1 5^1$ or $(4x+3)^1 (4x+4)^1 4^1$.
  Player 2 removes 1 coin from the opposing pile 
  to obtain $(4x+2)^1 (4x+4)^1 4^1$, 
  which is a $\mathcal{P}$-position by 
  Theorem \ref{thm:ThreeHeapab4Pposition}.
Therefore, $(4x+3)^1 (4x+4)^1 5^1$ is a 
$\mathcal{P}$-position.

By Remark \ref{rem:ShifetedBitWiseSum}, $b=\varphi_5(a)=a\oplus_2 3$
 is equivalent to $(a+2)\oplus (b+2)=3$.\qed
\end{proof}

\begin{table}[t]
\centering
\begin{tabular}{ccccccccccccccc}
\multicolumn{10}{c}{List of games $a^1 b^1 6^1$ that are
$\mathcal{P}$-positions} \\
\hline
(0,8) &(1,9)  &(2,5)  &(3,6)  &(4,7)   &(5,2)
&(6,3)  &(7,4)  &(8,0)  &(9,1)  
\\[0.025in]
(10,16) &(11,17)  &(12,18)  &(13,19)  &(14,20)  &(15,21)
&(16,10)  &(17,11)  &(18,12)  &(19,13)  
\\[0.025in]
(20,14) &(21,15)  &(22,24)  &(23,25)  &(24,22)   &(25,23)
&(26,28)  &(27,29)  &(28,26)  &(29,27)  
\\[0.025in]
(30,34) &(31,35)  &(32,36)  &(33,37)  &(34,30)   &(35,31)
&(36,32)  &(37,33)  &(38,42)  &(39,43)  
\\[0.025in]
(40,44) &(41,45)  &(42,38)  &(43,39)  &(44,40)   &(45,41)
&(46,50)  &(47,51)  &(48,52)  &(49,53)  
\\[0.025in]
(50,46) &(51,47)  &(52,48)  &(53,49)  &(54,58)   &(55,59)
&(56,60)  &(57,61)  &(58,54)  &(59,55)  
\\[0.025in]
(60,56) &(61,57)  &(62,66)  &(63,67)  &(64,68)   &(65,69)
&(66,62)  &(67,63)  &(68,64)  &(69,65)  
\\[0.025in]
(70,74) &(71,75)  &(72,76)  &(73,77)  &(74,70)   &(75,71)
&(76,72)  &(77,73)  &(78,82)  &(79,83)  
\\[0.025in]
\end{tabular} 
\caption{$\mathcal{P}$-position games $a^1 b^1 6^1$ 
listed as $(a,b)$.}
\label{tab:ab6Ppoisitions}
\end{table}

The $\mathcal{P}$-positions of the three-pile
games $a^1 \, b^1 \, 6^1$, for $0\le a \le 79$,
are listed in Table \ref{tab:ab6Ppoisitions}
as the ordered pair $(a,b)$.
We characterize the three-pile games $a^1 b^1 6^1$ that
are $\mathcal{P}$-positions in Theorem \ref{thm:ThreeHeapab6Pposition}.
In the next two theorems, we need to verify that the game
$(4x-1+y)^1 (4x+1+y)^1 3^1$ is an $\N$-position.

\begin{lemma}\label{lem:HalveCase2}
    Suppose $x\ge 4$ and $y\in\{0,\, 1\}$. Then the game $(4x-1+y)^1 (4x+1+y)^1 3^1$ is an $\N$-position.
\end{lemma}

\begin{proof}
    First, assume $x\ge 8$.
 Since $4x-1+y$ and $4x+1+y$ have the same parity,
 $(4x-1+y)\oplus (4x+1+7)$ is even.
 Thus, $(4x-1+y)\oplus (4x+1+y)\not\in\{1,\, 3\}$.
 By Theorem \ref{thm:ThreeHeapab3Pposition},
 $(4x-1+y)^1 (4x+1+y)^1 3^1$ is an $\mathcal{N}$-position.

Second, assume $4\le x \le 7$.
If $x\in\{4,\,6\}$ and $y=0$, Player 1 applies the halving operation to
 $(4x-1)^1 (4x+1)^1 3^1$ to obtain  $(2x-1)^1 (2x)^1 1^1$.
 If $x\in\{5,\, 7\}$ and $y=0$, Player 1 removes 3 coins from the pile 
 of height $4x-1$ to obtain  $(4x-4)^1 (4x+1)^1 3^1$.
 If $x\in\{4,\, 6\}$ and $y=1$, Player 1 removes 1 coin from the pile 
 of height $4x$ to obtain  $(4x-1)^1 (4x+2)^1 3^1$.
  If $x\in\{5,\, 7\}$ and $y=1$, Player 1 removes 5 coins from the pile 
 of height $4x+2$ to obtain  $(4x)^1 (4x-3)^1 3^1$.
 One observes from Tables \ref{tab:ab1Ppoisitions} and
 \ref{tab:ab3Ppoisitions} that each of the Player 1 moves results 
 in a $\mathcal{P}$-position.
\end{proof}

\begin{theorem}\label{thm:ThreeHeapab6Pposition}
The first 30 values of $\varphi_6$, in the form of 
$(a,\varphi_6(a))$, are displayed in Table \ref{tab:ab6Ppoisitions}.
    Suppose $x\ge 30$. 
       Then 
       \begin{equation*}
        \varphi_6(x)= x \oplus_2 4.   
       \end{equation*}
   Hence, $a^1 b^1 6^1$ is a $\mathcal{P}$-position if and only if
   $b=\varphi_6(a)$.
   Furthermore, for $a,b\ge 30$,
   $a^1 b^1 6^1$ is a $\mathcal{P}$-position if and only if
   $(a+2)\oplus (b+2)=4$.
\end{theorem}

\begin{proof}
It is straight forward to verify that, for 
$0\le x < 30$, $x^1 y^1 6^1$ 
is a $\mathcal{P}$-position if and only if
$y=\varphi_6(x)$. We leave the details of the 
proof to the reader.
This establishes the base case.

We want to show that $(8x-2+r)^1 (8x+2+r)^1 6^1$ is a
    $\mathcal{P}$-position for $x\ge 4$ and $r\in\{0,\, 1\}$.
We assume that for $y<8x-2$, $y^1 z^1 6^1$ is a 
$\mathcal{P}$-position if and only if $z=\varphi_6(y)$.
We observe that for all $y<8x-2$,
  $\varphi_6(y)<8x-2$.
Suppose Player 1 applies a standard Nim move to $(8x-2+r)^1 (8x+2+r)^1 6^1$.
By Remark \ref{rem:BeginInduction}, we may assume that 
the resulting position is
 either $y^1 (8x+2+r)^1 6^1$ for $8x-2\le y <8x-2+r$,
 $(8x-2+r)^1 y^1 6^1$ for $8x-2\le y <8x+2+r$, 
  or $(8x-2+r)^1 (8x+2+r)^1 y^1$ for $0\le y<6$.

Assume $r=0$.
Suppose, for $y\in\{1,\, 2\}$, Player 1 removes $y$ coins from the pile
of height either $8x+2$ or $6$ to obtain $(8x-2)^1 (8x+2-y)^1 6^1$ or
$(8x-2)^1 (8x+2)^1 (6-y)^1$. Player 2 removes $y$ coins from the opposing
pile to obtain $(8x-2)^1 (8x+2-y)^1 (6-y)^1$, which is a 
$\mathcal{P}$-position by Theorems \ref{thm:ThreeHeapab4Pposition}
and \ref{thm:ThreeHeapab5Pposition}.

Assume $r\in\{0,\, 1\}$.
Suppose, for $y\in\{3,\, 4\}$, Player 1 removes $y+r$ coins from the pile
of height $8x+2+r$ to obtain $(8x-2+r)^1 (8x+2-y)^1 6^1$. 
Since $(8x-2+r)\oplus (8x+3-y)\in\{0,\, 1\}$, $(8x-2)^1 (8x+2-y)^1 6^1$
is an $\mathcal{N}$-position by Lemma \ref{lem:NPositions}.
Suppose, for $3\le y\le 6$,  Player 1 removes $y$ coins from the pile of
height $6$ to obtain $(8x-2+r)^1 (8x+2+r)^1 (6-y)^1$.
Since $8x-2+r<8x<8x+2+r$,
the binary representations of $8x-2+r$ and $8x+2+r$
differ in the $2^3$ place value.
Thus, $(8x-2+r)\oplus (8x+2+r)\ge 8$.
By Lemma \ref{lem:NPositions2}, $(8x-2+r)^1 (8x+2+r)^1 (6-y)^1$ is an $\N$-position.

 Suppose Player 1 applies the halving operation 
  to $(8x-2+r)^1 (8x+2+r)^1 6^1$
  to obtain $(4x-1+r_1)^1 (4x+1+r_1)^1 3^1$ where $r_1\in\{0,\, 1\}$.
By Lemma \ref{lem:HalveCase2}, $(4x-1+r_1)^1 (4x+1+r_1)^1 3^1$
is an $\N$-position.
Hence, $(8x-2)^1 (8x+2)^1 6^1$ is a $\mathcal{P}$-position.

Assume $r=1$.
We show $(8x-1)^1 (8x+3)^1 6^1$ is a $\mathcal{P}$-position.
Suppose Player 1 removes 1 coin from the pile of height either 
$8x-1$ or $8x+3$ to obtain $(8x-2)^1 (8x+3)^1 6^1$ 
or $(8x-1)^1 (8x+2)^1 6^1$.
Player 2 removes $1$ coin from the opposing
pile to obtain $(8x-2)^1 (8x+2)^1 6^1$, which is a 
$\mathcal{P}$-position.
Suppose, for $y\in\{2,\, 3\}$, Player 1 removes $y$ coins from the pile
of height $8x+3$ to obtain $(8x-1)^1 (8x+3-y)^1 6^1$. 
Player 2 removes $4-y$ coins from the pile of height $6$ to obtain
$(8x-1)^1 (8x+3-y)^1 (2+y)^1$, which is a $\mathcal{P}$-position by
Theorems \ref{thm:ThreeHeapab4Pposition}
and \ref{thm:ThreeHeapab5Pposition}.
Suppose, for $y\in\{1,\, 2\}$, Player 1 removes $y$ coins from the pile
of height $6$ to obtain $(8x-1)^1 (8x+3)^1 (6-y)^1$. 
Player 2 removes $4-y$ coins from the pile of height $8x+3$ to obtain
$(8x-1)^1 (8x+y-1)^1 (6-y)^1$, which is a $\mathcal{P}$-position by
Theorems \ref{thm:ThreeHeapab4Pposition}
and \ref{thm:ThreeHeapab5Pposition}.
Hence, $(8x-1)^1 (8x+3)^1 6^1$ is a $\mathcal{P}$-position.

The proof that $(8x-2+r)^1 (8x+2+r)^1 6^1$ is a
    $\mathcal{P}$-position, for $x\ge 4$ and $r\in\{2,\, 3\}$,
    is similar to the cases when $x\ge 4$ and $r\in\{0,\, 1\}$.
    We leave the details to the reader.
    
By Remark \ref{rem:ShifetedBitWiseSum}, $b=\varphi_6(a)=a\oplus_2 4$
 is equivalent to $(a+2)\oplus (b+2)=4$.\qed
\end{proof}

\begin{table}[t!]
\centering
\begin{tabular}{ccccccccccccccc}
\multicolumn{10}{c}{List of games $a^1 b^1 7^1$ that are
$\mathcal{P}$-positions} \\
\hline
(0,9) &(1,8)  &(2,7)  &(3,5)  &(4,6)   &(5,3)
&(6,4)  &(7,2)  &(8,1)  &(9,0)  
\\[0.025in]
(10,17) &(11,16)  &(12,19)  &(13,18)  &(14,21)  &(15,20)
&(16,11)  &(17,10)  &(18,13)  &(19,12)  
\\[0.025in]
(20,15) &(21,14)  &(22,25)  &(23,24)  &(24,23)   &(25,22)
&(26,29)  &(27,28)  &(28,27)  &(29,26)  
\\[0.025in]
(30,35) &(31,34)  &(32,37)  &(33,36)  &(34,31)   &(35,30)
&(36,33)  &(37,32)  &(38,43)  &(39,42)  
\\[0.025in]
(40,45) &(41,44)  &(42,39)  &(43,38)  &(44,41)   &(45,40)
&(46,51)  &(47,50)  &(48,53)  &(49,52)  
\\[0.025in]
(50,47) &(51,46)  &(52,49)  &(53,48)  &(54,59)   &(55,58)
&(56,61)  &(57,60)  &(58,55)  &(59,54) 
\\[0.025in]
(60,57) &(61,56)  &(62,67)  &(63,66)  &(64,69)   &(65,68)
&(66,63)  &(67,62)  &(68,65)  &(69,64)  
\\[0.025in]
(70,75) &(71,74)  &(72,77)  &(73,76)  &(74,71)   &(75,70)
&(76,73)  &(77,72)  &(78,83)  &(79,82)  
\\[0.025in]
\end{tabular} 
\caption{$\mathcal{P}$-position games $a^1 b^1 7^1$ 
listed as $(a,b)$.}
\label{tab:ab7Ppoisitions}
\end{table}

\begin{theorem}\label{thm:ThreeHeapab7Pposition}
The first 30 values of $\varphi_7$, in the form of 
$(a,\varphi_7(a))$, are displayed in Table \ref{tab:ab7Ppoisitions}.
Suppose $x\ge 30$. Then 
       \begin{equation*}
        \varphi_7(x)= x \oplus_2 5.   
       \end{equation*}
   Hence, $a^1 b^1 7^1$ is a $\mathcal{P}$-position if and only if
   $b=\varphi_7(a)$.
   Furthermore, for $a,b\ge 30$,
   $a^1 b^1 7^1$ is a $\mathcal{P}$-position if and only if
   $(a+2)\oplus (b+2)=5$.
\end{theorem}

\begin{proof}
It is straight forward to verify that, for 
$0\le x < 30$, $x^1 y^1 7^1$ 
is a $\mathcal{P}$-position if and only if
$y=\varphi_7(x)$. We leave the details of the 
proof to the reader.
This establishes the base case.

We want to show that $(8x-2)^1 (8x+3)^1 7^1$ is a $\mathcal{P}$-position for $x\ge 4$.
We assume that for $y<8x-2$, $y^1 z^1 7^1$ is a 
$\mathcal{P}$-position if and only if $z=\varphi_7(y)$.
We observe that for all $y<8x-2$,
  $\varphi_7(y)<8x-2$.
Suppose Player 1 applies a standard Nim move to $(8x-2)^1 (8x+3)^1 7^1$.
By Remark \ref{rem:BeginInduction}, we may assume that 
the resulting position is
 either $(8x-2)^1 y^1 7^1$ for $8x-2\le y <8x+3$ 
  or $(8x-2)^1 (8x+3)^1 y^1$ for $0\le y<7$.
  
Suppose, for $1\le y\le 3$, Player 1 removes $y$ coins from the pile
of height either $8x+3$ or $7$ to obtain 
$(8x-2)^1 (8x+3-y)^1 7^1$ or $(8x-2)^1 (8x+3)^1 (7-y)^1$. 
Player 2 removes $y$ coins from the opposing pile to obtain
$(8x-2)^1 (8x+3-y)^1 (7-y)^1$, which is a $\mathcal{P}$-position by
Theorems \ref{thm:ThreeHeapab4Pposition}, 
\ref{thm:ThreeHeapab5Pposition}, and \ref{thm:ThreeHeapab6Pposition}.
Suppose, for $y\in\{4,\, 5\}$, Player 1 removes $y$ coins from the pile
of height $8x+3$ to obtain $(8x-2)^1 (8x+3-y)^1 7^1$. 
Since $(8x-2)\oplus (8x+3-y)\in\{0,\, 1\}$, $(8x-2)^1 (8x+3)^1 7^1$
is an $\mathcal{N}$-position by Lemma \ref{lem:NPositions}.
  Suppose, for $4\le y\le 7$,  Player 1 removes $y$ coins from the pile of
height $7$ to obtain $(8x-2)^1 (8x+3)^1 (7-y)^1$.
Since $8x-2<8x<8x+3$,
the binary representations of $8x-2$ and $8x+3$
differ in the $2^3$ place value.
Thus, $(8x-2)\oplus (8x+3)\ge 8$.
By Lemma \ref{lem:NPositions2}, $(8x-2)^1 (8x+3)^1 (7-y)^1$ is an $\N$-position.
  
Suppose Player 1 applies the halving operation to  
$(8x-2)^1 (8x+3)^1 7^1$ to obtain $(4x-1)^1 (4x+1)^1 3^1$.
By Lemma \ref{lem:HalveCase2}, $(4x-1)^1 (4x+1)^1 3^1$
is an $\N$-position.
 Therefore, $(8x-2)^1 (8x+3)^1 7^1$ is a $\mathcal{P}$-position.

The proof that $(8x-1)^1 (8x+2)^1 7^1$, for $x\ge 4$, 
and $(8x+r)^1 (8x+5-r)^1 7^1$, for $x\ge 4$ and $r\in\{0,\, 1\}$, are
    $\mathcal{P}$-positions
    is similar to the proof that
    $(8x-2)^1 (8x+3)^1 7^1$, for $x\ge 4$, is a $\mathcal{P}$-position.
    We leave the details to the reader.
    
By Remark \ref{rem:ShifetedBitWiseSum}, $b=\varphi_7(a)=a\oplus_2 5$
 is equivalent to $(a+2)\oplus (b+2)=5$.\qed
\end{proof}

\begin{corollary}\label{cor:BitShift2}
    We have $\varphi_c(x)= x \oplus_2  (c-2)$
    for $c\in\{4,\, 5,\, 6,\, 7\}$ and $x\ge 30$.
\end{corollary}

\begin{corollary}\label{cor:BitShiftSum}
    Let $a,b\ge 30$ and $4\le c\le 7$. Then the following statements are equivalent.
    \begin{enumerate}
        \item The three pile game $a^1 b^1 c^1$ is a $\mathcal{P}$-position.
        \item $a = b \oplus_2 (c-2)$.
        \item $(a+2)\oplus (b+2)=c-2$.
    \end{enumerate}
\end{corollary}

We are in a position to prove the main result of this paper.

\begin{proof}[Theorem \ref{thm:MainResult}]
    Suppose $0\le c\le 3$. From \eqref{eqn:0PileResult}, 
    \eqref{eqn:1PileResult}, \eqref{eqn:2PileResult}, and
    \eqref{eqn:3PileResult}, $a^1 b^1 c^1$
    is a $\P$-position if and only if
    \begin{equation*}
        a\oplus b= (1+\epsilon_1)+(-1)^{\lfloor \log_2(a/7)\rfloor +\epsilon_2},
    \end{equation*}
    where $\epsilon_1=0$ if $c\in\{0,\, 2\}$,  $\epsilon_1=1$ if $c\in\{1,\, 3\}$,
     $\epsilon_2=0$ if $c\in\{0,\, 1\}$, and  $\epsilon_2=1$ if $c\in\{2,\,  3\}$.
     We observe that $\epsilon_1=\tfrac{1}{2} ( 1+(-1)^{c+1} )$ and 
      $\epsilon_2=\tfrac{1}{2} ( 1+(-1)^{\lfloor c/2\rfloor +1} )$.
      Thus, \eqref{eqn:FirstBitsum} holds.
      Also, \eqref{eqn:SecondBitsum} is an immediate consequence of Corollary \ref{cor:BitShiftSum}.
\end{proof}

We use Corollary \ref{cor:BitShiftSum} to determine 
various three-pile games of Halve Nim that are $\N$-positions.
We list these three-pile games  in the following two lemmas.

\begin{lemma}\label{lem:MoreNPositions}
Let $a$, $b$, and $c$ be integers such that $a,b\ge 30$ and $c\ge 8$.
If $(a+2)\oplus (b+2)\in\{2,\, 3,\, 4,\, 5\}$, then the three-pile game
$a^1 b^1 c^1$ is an $\mathcal{N}$-position.
\end{lemma}

\begin{proof}
    
    Suppose $(a+2)\oplus (b+2)=k-2$ for $k\in\{4,\, 5,\, 6,\, 7\}$.
    Player 1 removes $c-k$ coins from
    the pile of height $c$ to obtain $a^1 b^1 k^1$, which is a
    $\mathcal{P}$-position by 
    Corollary \ref{cor:BitShiftSum}.\qed
\end{proof}

\begin{lemma}\label{lem:MoreNPositions2}
Let $a$, $b$, and $c$ be integers such that $a,b\ge 30$ and 
$4\le c\le 7$.
If $(a+2)\oplus (b+2)\not\in\{2,\, 3,\, 4,\, 5\}$, 
then the three-pile game
$a^1 b^1 c^1$ is an $\mathcal{N}$-position.
\end{lemma}

\begin{proof}
    By Corollary \ref{cor:BitShiftSum}, if $a^1 b^1 c^1$
    is a $\P$-position, then $(a+2)\oplus (b+2)\in\{2,\, 3,\, 4,\, 5\}$.
    Thus, \, if $(a+2)\oplus (b+2)\not\in\{2,\, 3,\, 4,\, 5\}$,
   then $a^1 b^1 c^1$ is an $\mathcal{N}$-position.\qed
\end{proof}

\begin{table}[t!]
\centering
\begin{tabular}{ccccccccccccccc}
\multicolumn{10}{c}{List of games $a^1 b^1 8^1$ that are
$\mathcal{P}$-positions} \\
\hline
(0,6) &(1,7)  &(2,8)  &(3,9)  &(4,10)   &(5,11)
&(6,0)  &(7,1)  &(8,2)  &(9,3)  
\\[0.025in]
(10,4) &(11,5)  &(12,16)  &(13,17)  &(14,22)  &(15,23)
&(16,12)  &(17,13)  &(18,20)  &(19,21)  
\\[0.025in]
(20,18) &(21,19)  &(22,14)  &(23,15)  &(24,30)   &(25,31)
&(26,32)  &(27,33)  &(28,34)  &(29,35)  
\\[0.025in]
(30,24) &(31,25)  &(32,26)  &(33,27)  &(34,28)   &(35,29)
&(36,40)  &(37,41)  &(38,44)  &(39,45)  
\\[0.025in]
(40,36) &(41,37)  &(42,48)  &(43,49)  &(44,38)   &(45,39)
&(46,52)  &(47,53)  &(48,42)  &(49,43)  
\\[0.025in]
(50,56) &(51,57)  &(52,46)  &(53,47)  &(54,60)   &(55,61)
&(56,50)  &(57,51)  &(58,62)  &(59,63)  
\\[0.025in]
(60,54) &(61,55)  &(62,58)  &(63,59)  &(64,70)   &(65,71)
&(66,72)  &(67,73)  &(68,74)  &(69,75)  
\\[0.025in]
(70,64) &(71,65)  &(72,66)  &(73,67)  &(74,68)   &(75,69)
&(76,82)  &(77,83)  &(78,84)  &(79,85) 
\\[0.025in]
\end{tabular} 
\caption{$\mathcal{P}$-position games $a^1 b^1 8^1$ 
listed as $(a,b)$.}
\label{tab:ab8Ppoisitions}
\end{table}

\begin{definition}\label{defn:ListShiftBitwiseSum}
    Let $t$ be a positive integer and let $S=(s_0,s_1,\ldots,s_{n-1})$ be
    a sequence of $n$ integers.
    The \emph{list shifted bitwise sum by $S$  with threshold $t$} is given by
    \begin{equation*}
       x \leftidx{_t}{\oplus}{_S} y = x \oplus_{s_{x-t}} y
    \end{equation*}
   where the index $x-t$ of $s_{x-t}$ is taken modulo $n$,
   and $x\oplus_{s_{x-t}} y$ is the $s_{x-t}$-shift bitwise sum
   (See Definition \ref{defn:ShiftedBitwiseSum}).
\end{definition}

In the next two theorems, we need to verify that the game
$(6x+2)^1 (6x+5)^1 4^1$ is an $\N$-position.

\begin{lemma}\label{lem:HalveCase3}
    Suppose $x\ge 5$. Then the game $(6x+2)^1 (6x+5)^1 4^1$ is an $\N$-position.
\end{lemma}

\begin{proof}
 Since $6x+4$ and $6x+7$ have the opposite parity,
 $(6x+4)\oplus (6x+7)$ is odd.
  Hence, $\big( (6x+2)+2\big)\oplus \big( (6x+5)+2\big) \ne 2$.
  By Corollary \ref{cor:BitShiftSum}, $(6x+2)^1 (6x+5)^1 4^1$
  is an $\N$-position.\qed
\end{proof}

\begin{theorem}\label{thm:ThreeHeapab8Pposition}
The first 64 values of $\varphi_8$, in the form of 
$(a,\varphi_8(a))$, are displayed in Table \ref{tab:ab8Ppoisitions}.
Let $S_8=\big( (8,7,6,5,4,3)^2, (4,3,2,1,0,-1)^2\big)$.
For $x\ge 64$, we have
\begin{equation*}
    \varphi_8(x)=x \; \leftidx{_{64}}{\oplus}{_{S_8}} \, 6
\end{equation*}
   Hence, $a^1 b^1 8^1$ is a $\mathcal{P}$-position if and only if
   $b=\varphi_8(a)$.
\end{theorem}

\begin{proof}
It is straight forward to verify that, for 
$0\le x < 64$, $x^1 y^1 8^1$ 
is a $\mathcal{P}$-position if and only if
$y=\varphi_8(x)$. We leave the details of the 
proof to the reader.
This establishes the base case.

We show that $(12x+4)^1 (12x+10)^1 8^1$ is 
a $\mathcal{P}$-position  for $x\ge 5$.
We assume that for $y<12x+4$, $y^1 z^1 8^1$ is a $\mathcal{P}$-position
if and only if $z=\varphi_8(y)$.
We observe that for all $y<12x+4$,
  $\varphi_8(y)<12x+4$.
Suppose Player 1 applies a standard Nim move to  $(12x+4)^1 (12x+10)^1 8^1$.
By Remark \ref{rem:BeginInduction}, we may assume that 
the resulting position is
 either $(12x+4)^1 y^1 8^1$ for $12x+4\le y <12x+10$
  or $(12x+4)^1 (12x+10)^1 y^1$ for $0\le y<8$.

 Suppose, for $y\in\{1,\, 2\}$, Player 1 removes $y$ coins from the pile of
  height $12x+10$ to obtain $(12x+4)^1 (12x+10-y)^1 8^1$.
  If $x$ is odd,
  Player 2 removes $y$ coins from the pile of height $8$
  to obtain $(12x+4)^1 (12x+10-y)^1 (8-y)^1$, 
  which is a $\mathcal{P}$-position 
  by Theorems \ref{thm:ThreeHeapab6Pposition} and 
  \ref{thm:ThreeHeapab7Pposition}.
   If $x$ is even,
  Player 2 applies the halving operation 
  to $(12x+4)^1 (12x+10-y)^1 8^1$ to obtain 
  $(6x+2)^1 (6x+4)^1 4^1$,
  which is a $\mathcal{P}$-position 
  by Theorem \ref{thm:ThreeHeapab4Pposition}.
  Suppose, for $3\le y\le 6$, Player 1 removes $y$ coins from the pile of
  height $12x+10$ to obtain $(12x+4)^1 (12x+10-y)^1 8^1$.
Since $(12x+4)\oplus (12x+10-y)=6-y\in\{0,\, 1,\, 2,\, 3\}$ for $3\le y\le 6$,
 $(12x+4)^1 (12x+10-y)^1 8^1$ is an $\N$-position 
 by Lemma \ref{lem:NPositions}.

  Suppose, for $1\le y\le 4$, Player 1 removes $y$ coins from
  the pile of height 8 to obtain $(12x+4)^1 (12x+10)^1 (8-y)^1$.
  Since $(12x+4)+2\equiv 2 \text{ or }6\pmod{8}$,
  there exists an integer $q$ such that $8q-6\le 12x+6<8q$.
  So, the binary
  representations of $(12x+4)+2$ and $(12x+10)+2$
differ in the $2^3$ place value.
Thus, $\big( (12x+4)+2\big)\oplus \big( (12x+10)+2\big)\ge 8$.
By Lemma \ref{lem:MoreNPositions2}, 
$(12x+4)^1 (12x+10)^1 (8-y)^1$ is an $\N$-position.
Suppose, for $5\le y\le 8$, Player 1 removes $y$ coins from
  the pile of height 8 to obtain $(12x+4)^1 (12x+10)^1 (8-y)^1$.
  Since $2^2<6<2^3$, the binary
  representations of $12x+4$ and $12x+10$
differ in the $2^2$ or $2^3$ place value.
Thus, $(12x+4)\oplus (12x+10)\ge 4$.
By Lemma \ref{lem:NPositions2}, 
$(12x+4)^1 (12x+10)^1 (8-y)^1$ is an $\N$-position.

 Suppose Player 1 applies the halving operation to  $(12x+4)^1 (12x+10)^1 8^1$
to obtain $(6x+2)^1 (6x+5)^1 4^1$.
By Lemma \ref{lem:HalveCase3}, $(6x+2)^1 (6x+5)^1 4^1$
  is an $\N$-position.
 Therefore, $(12x+4)^1 (12x+10)^1 8^1$ is a $\mathcal{P}$-position.

The proof that, for $x\ge 5$ and $1\le r\le 5$,
$(12x+4+r)^1 (12x+10+r)^1 8^1$
is a $\mathcal{P}$-position
    is similar to the proof that
    $(12x+4)^1 (12x+10)^1 8^1$ is a $\mathcal{P}$-position.
    We leave the details to the reader.
\end{proof}

\begin{table}[t!]
\centering
\begin{tabular}{ccccccccccccccc}
\multicolumn{10}{c}{List of games $a^1 b^1 9^1$ that are
$\mathcal{P}$-positions} \\
\hline
(0,7) &(1,6)  &(2,9)  &(3,8)  &(4,11)   &(5,10)
&(6,1)  &(7,0)  &(8,3)  &(9,2)  
\\[0.025in]
(10,5) &(11,4)  &(12,17)  &(13,16)  &(14,23)  &(15,22)
&(16,13)  &(17,12)  &(18,21)  &(19,20)  
\\[0.025in]
(20,19) &(21,18)  &(22,15)  &(23,14)  &(24,31)   &(25,30)
&(26,33)  &(27,32)  &(28,35)  &(29,34)  
\\[0.025in]
(30,25) &(31,24)  &(32,27)  &(33,26)  &(34,29)   &(35,28)
&(36,41)  &(37,40)  &(38,45)  &(39,44)  
\\[0.025in]
(40,37) &(41,36)  &(42,49)  &(43,48)  &(44,39)   &(45,38)
&(46,53)  &(47,52)  &(48,43)  &(49,42)  
\\[0.025in]
(50,57) &(51,56)  &(52,47)  &(53,46)  &(54,61)   &(55,60)
&(56,51)  &(57,50)  &(58,63)  &(59,62)  
\\[0.025in]
(60,55) &(61,54)  &(62,59)  &(63,58)  &(64,71)   &(65,70)
&(66,73)  &(67,72)  &(68,75)  &(69,74)  
\\[0.025in]
(70,65) &(71,64)  &(72,67)  &(73,66)  &(74,69)   &(75,68)
&(76,83)  &(77,82)  &(78,85)  &(79,84) 
\\[0.025in]
\end{tabular} 
\caption{$\mathcal{P}$-position games $a^1 b^1 9^1$ 
listed as $(a,b)$.}
\label{tab:ab9Ppoisitions}
\end{table}

\begin{theorem}\label{thm:ThreeHeapab9Pposition}
The first 64 values of $\varphi_9$, in the form of 
$(a,\varphi_9(a))$, are displayed in Table \ref{tab:ab9Ppoisitions}.
Let $S_9=\big( (8^2,6^2,4^2)^2, (4^2,2^2,0^2)^2\big)$.
For $x\ge 64$, we have
\begin{equation*}
    \varphi_9(x)=x \; \leftidx{_{64}}{\oplus}{_{S_{9}}} \, 7
\end{equation*}
   Hence, $a^1 b^1 9^1$ is a $\mathcal{P}$-position if and only if
   $b=\varphi_9(a)$.
\end{theorem}

\begin{proof}
It is straight forward to verify that, for 
$0\le x < 64$, $x^1 y^1 9^1$ 
is a $\mathcal{P}$-position if and only if
$y=\varphi_9(x)$. We leave the details of the 
proof to the reader.
This establishes the base case.

We show that $(12x+4)^1 (12x+11)^1 9^1$ is 
a $\mathcal{P}$-position for $x\ge 5$.
We assume that for $y<12x+4$, $y^1 z^1 9^1$ is a $\mathcal{P}$-position
if and only if $z=\varphi_9(y)$.
We observe that for all $y<12x+4$,
  $\varphi_9(y)<12x+4$.
Suppose Player 1 applies a standard Nim move to  $(12x+4)^1 (12x+11)^1 9^1$.
By Remark \ref{rem:BeginInduction}, we may assume that 
the resulting position is
 either $(12x+4)^1 y^1 9^1$ for $12x+4\le y <12x+11$
  or $(12x+4)^1 (12x+11)^1 y^1$ for $0\le y<9$.
  
 Suppose Player 1 removes 1 coin from the pile of
  height either $12x+11$ or $9$ to obtain $(12x+4)^1 (12x+10)^1 9^1$
  or  $(12x+4)^1 (12x+11)^1 8^1$.
  Player 2 removes 1 coin from the opposing pile to obtain 
  $(12x+4)^1 (12x+10)^1 8^1$, which is a $\P$-position
  by Theorem \ref{thm:ThreeHeapab8Pposition}.
   Suppose, for $y\in\{2,\, 3\}$,
   Player 1 removes $y$ coins from the pile of
  height $12x+11$ to obtain $(12x+4)^1 (12x+11-y)^1 9^1$.
   If $x$ is odd,
  Player 2 removes $y$ coins from the pile of height $9$ 
  to obtain $(12x+4)^1 (12x+11-y)^1 (9-y)^1$,
  which is a $\mathcal{P}$-position 
  by Theorems \ref{thm:ThreeHeapab6Pposition} and 
  \ref{thm:ThreeHeapab7Pposition}.
  If $x$ is even,
  Player 2 applies the halving operation to $(12x+4)^1 (12x+11-y)^1 9^1$
  to obtain the position
  $(6x+2)^1 (6x+4)^1 4^1$,
  which is a $\mathcal{P}$-position 
  by Theorem \ref{thm:ThreeHeapab4Pposition}.
  Suppose, for $4\le y \le 7$, 
  Player 1 removes $y$ coins from the pile of
  height $12x+11$ to obtain $(12x+4)^1 (12x+11-y)^1 9^1$.
Since $(12x+4)\oplus (12+11-y)=7-y\in\{0,\, 1,\, 2,\, 3\}$,
 $(12x+4)^1 (12x+4+y)^1 9^1$ is an $\N$-position 
 by Lemma \ref{lem:NPositions}.

  Suppose, for $2\le y\le 5$, Player 1 removes $y$ coins from
  the pile of height 9 to obtain $(12x+4)^1 (12x+11)^1 (9-y)^1$.
  Since $(12x+4)+2\equiv 2 \text{ or }6\pmod{8}$,
  there exists an integer $q$ such that $8q-7\le 12x+6<8q$.
  So, the binary
  representations of $(12x+4)+2$ and $(12x+11)+2$
differ in the $2^3$ place value.
Thus, $\big( (12x+4)+2\big)\oplus \big( (12x+11)+2\big)\ge 8$.
By Lemma \ref{lem:MoreNPositions2}, 
$(12x+4)^1 (12x+11)^1 (9-y)^1$ is an $\N$-position.
Suppose,  for $6\le y\le 9$, Player 1 removes $y$ coins from
  the pile of height 9 to obtain $(12x+4)^1 (12x+11)^1 (9-y)^1$.
  Since $2^2<7<2^3$, the binary
  representations of $12x+4$ and $12x+11$
differ in the $2^2$ or $2^3$ place value.
Thus, $(12x+4)\oplus (12x+11)\ge 4$.
By Lemma \ref{lem:NPositions2}, 
$(12x+4)^1 (12x+11)^1 (9-y)^1$ is an $\N$-position.

 Suppose Player 1 applies the halving operation to  $(12x+4)^1 (12x+11)^1 9^1$
 to obtain $(6x+2)^1 (6x+5)^1 4^1$.
By Lemma \ref{lem:HalveCase3}, $(6x+2)^1 (6x+5)^1 4^1$
  is an $\N$-position.
 Therefore, $(12x+4)^1 (12x+11)^1 9^1$ is a $\mathcal{P}$-position.

The proof that $(12x+5)^1 (12x+10)^1 9^1$, for $x\ge 5$, and
$(12x+4+2y+r)^1 (12x+11+2y-r)^1 9^1$,
    for $x\ge 5$, $y\in\{1,\, 2\}$, and $r\in\{0,\, 1\}$,
  are $\mathcal{P}$-positions
    is similar to the proof that
    $(12x+4)^1 (12x+11)^1 9^1$, for $x\ge 5$, is a $\mathcal{P}$-position.
    We leave the details to the reader.
\end{proof}

\begin{corollary}\label{cor:BitShiftSeq}
    Let $c\in\{8,9\}$ and $x\ge 64$. Then
    \begin{equation*}
        \varphi_c(x) = x \; \leftidx{_{64}}{\oplus}{_{S_c}} \, (c-2).
    \end{equation*}
\end{corollary}

\begin{table}[t!]
\centering
\begin{tabular}{ccccccccccccccc}
\multicolumn{10}{c}{List of games $a^1 b^1 10^1$ that are
$\mathcal{P}$-positions} \\
\hline
(0,12) &(1,13)  &(2,10)  &(3,11)  &(4,8)   &(5,9)
&(6,16)  &(7,17)  &(8,4)  &(9,5)  
\\[0.025in]
(10,2) &(11,3)  &(12,0)  &(13,1)  &(14,24)  &(15,25)
&(16,6)  &(17,7)  &(18,22)  &(19,23)  
\\[0.025in]
(20,26) &(21,27)  &(22,18)  &(23,19)  &(24,14)   &(25,15)
&(26,20)  &(27,21)  &(28,32)  &(29,33)  
\\[0.025in]
(30,36) &(31,37)  &(32,28)  &(33,29)  &(34,38)   &(35,39)
&(36,30)  &(37,31)  &(38,34)  &(39,35)  
\\[0.025in]
(40,48) &(41,49)  &(42,46)  &(43,47)  &(44,50)   &(45,51)
&(46,42)  &(47,43)  &(48,40)  &(49,41)  
\\[0.025in]
(50,44) &(51,45)  &(52,56)  &(53,57)  &(54,62)   &(55,63)
&(56,52)  &(57,53)  &(58,64)  &(59,65)  
\\[0.025in]
(60,68) &(61,69)  &(62,54)  &(63,55)  &(64,58)   &(65,59)
&(66,70)  &(67,71)  &(68,60)  &(69,61)  
\\[0.025in]
(70,66) &(71,67)  &(72,78)  &(73,79)  &(74,80)   &(75,81)
&(76,84)  &(77,85)  &(78,72)  &(79,73) 
\\[0.025in]
\end{tabular} 
\caption{$\mathcal{P}$-position games $a^1 b^1 10^1$ 
listed as $(a,b)$.}
\label{tab:ab10Ppoisitions}
\end{table}

\begin{theorem}\label{thm:ThreeHeapab10Pposition}
The first 66 values of $\varphi_{10}$, in the form of 
$(a,\varphi_{10}(a))$, are displayed in Table \ref{tab:ab10Ppoisitions}. Let
\begin{align*}
S_{10}=\big( & 4^2, (-8)^2, (-4)^2, 6^4, 8^2, (-6)^4, 4^2, (-8)^2, (-4)^2,
8^2, 4^2, 8^2, \\
  &(-4)^2, (-8)^2, 4^2, (-8)^2, (-4)^2, 8^2, 4^2, 8^2, (-4)^2, (-8)^2 \big).
\end{align*}
For $x\ge 66$, we have
\begin{equation*}
    \varphi_{10}(x)=x \, + s_{x-66} \text{ \  for }
\end{equation*}
where the indices are taken modulo $48$.
   Hence, $a^1 b^1 10^1$ is a $\mathcal{P}$-position if and only if
   $b=\varphi_{10}(a)$.
\end{theorem}

\begin{proof}
The first 66 values of $\varphi_{10}$, in the form of 
$(a,\varphi_{10}(a))$, are displayed in Table \ref{tab:ab10Ppoisitions}.
We leave the details of the 
proof to the reader.
This establishes the base case.

    We will show that $(48x+18+r)^1 (48x+22+r)^1 {10}^1$
    is a $\mathcal{P}$-position for $x\ge 1$ and $r\in\{0,\, 1\}$.
    We assume that for $y<48x+18$, $y^1 z^1 {10}^1$ is a 
    $\P$-position if and only if $z=\varphi_{10}(z)$.
We observe that for all $y<48x+18$,
  $\varphi_{10}(y)<48x+18$.
Suppose Player 1 applies a standard Nim move to $(48x+18+r)^1 (48x+22+r)^1 {10}^1$.
By Remark \ref{rem:BeginInduction}, we may assume that 
the resulting position is
 either  $y^1 (48x+22+r)^1 {10}^1$ for $48x+18\le y <48x+18+r$,
 $(48x+18+r)^1 y^1 {10}^1$ for $48x+18\le y <48x+22+r$
  or $(48x+18+r)^1 (48x+22+r)^1 y^1$ for $0\le y<10$.

  Suppose, for $y\in\{1,\, 2\}$, 
  Player 1 removes $y+r$ coins from the pile of height $48x+22+r$
  to obtain $(48x+18+r)^1 (48x+22-y)^1 {10}^1$.
We observe that
$\big((48x+18+r)+2\big)\oplus \big((48x+22-y)+2\big)\in\{2,\, 3\}$.
By Lemma \ref{lem:MoreNPositions}, $(48x+18+r)^1 (48x+22-y)^1 {10}^1$
is an $\N$-position.
Suppose, for $y\in\{3,\, 4\}$, 
Player 1 removes $y+r$ coins from the pile of height $48x+22+r$
to obtain $(48x+18+r)^1 (48x+22-y)^1 {10}^1$.
We observe that
$(48x+18+r)\oplus (48x+22-y)\in\{0,\, 1\}$.
By Lemma \ref{lem:NPositions}, $(48x+18+r)^1 (48x+22-y)^1 {10}^1$
is an $\N$-position.

Suppose Player 1 applies the halving operation to 
$(48x+18+r)^1 (48x+22+r)^1 {10}^1$ to obtain
$(24x+9)^1 (24x+11)^1 {5}^1$.
Since $(24x+9)\oplus (24x+11)=2\ne 3$,
$(24x+9)^1 (24x+11)^1 {5}^1$ is an $\N$-position
by Lemma \ref{lem:NPositions}.

Suppose Player 1 removes $1$ coin from
  the pile of height 10 to obtain $(48x+18+r)^1 (48x+22+r)^1 9^1$.
Player 2 removes $1+2r$ coins from the 
pile of height $48x+18+r$
to obtain $(48x+17-r)^1 (48x+22+r)^1 9^1$,
which is a $\P$-position by Theorem \ref{thm:ThreeHeapab9Pposition}.
Suppose Player 1 removes $2$ coins from
  the pile of height 10 to obtain $(48x+18+r)^1 (48x+22+r)^1 8^1$.
Player 2 removes $2$ coins from the 
pile of height $48x+18+r$
to obtain $(48x+16+r)^1 (48x+22+r)^1 8^1$,
which is a $\P$-position by Theorem \ref{thm:ThreeHeapab8Pposition}.
Suppose, for $3\le y\le 6$, Player 1 removes $y$ coins from
  the pile of height 10 to obtain $(48x+18+r)^1 (48x+22+r)^1 (10-y)^1$.
  Since $(48x+18+r)+2\equiv 4 \text{ or }5\pmod{8}$,
  there exists an integer $q$ such that $8q-4\le (48x+18+r)+2<8q$.
  So, the binary
  representations of $(48x+18+r)+2$ and $(48x+22+r)+2$
differ in the $2^3$ place value.
Thus, $\big( (48x+18+r)+2\big)\oplus \big( (48x+22+r)+2\big)\ge 8$.
By Lemma \ref{lem:MoreNPositions2}, 
$(48x+18+r)^1 (48x+22+r)^1 (10-y)^1$ is an $\N$-position.
Suppose, for $7\le y\le 10$, Player 1 removes $y$ coins from
  the pile of height 10 to obtain $(48x+18+r)^1 (48x+22+r)^1 (10-y)^1$.
Since $48x+18+r$ and $48x+22+r$ differ by 4,
the binary representations of $48x+18+r$ and $48x+22+r$
differ in the $2^2$ place value.
So, $(48x+18+r)\oplus (48x+22+r)\ge 4$. 
By Lemma \ref{lem:NPositions2}, 
 $(48x+18+r)^1 (48x+22+r)^1 (10-y)^1$ is an $\N$-position for $0\le y \le 3$.

Assume $r=1$.
Suppose Player 1 removes 1 coin from the pile of height 
either $48x+19$ or $48x+23$ to obtain 
$(48x+18)^1 (48x+23)^1 {10}^1$ or $(48x+19)^1 (48x+22)^1 {10}^1$.
Player 2 removes 1 coin from the opposing pile 
to obtain $(48x+18)^1 (48x+22)^1 {10}^1$, which 
is a $\mathcal{P}$-position.
Therefore, $(48x+19)^1 (48x+23)^1 {10}^1$ is a $\mathcal{P}$-position.

    The proof that $(48x+18+r)^1 (48x+18+r+s_{r})^1 {10}^1$
    is a $\mathcal{P}$-position, for $x\ge 1$ and $2\le r<48$,
    is similar to the proof that $(48x+18+r)^1 (48x+22+r)^1 {10}^1$
    is a $\mathcal{P}$-position for $x\ge 1$ and $0\le r\le 1$.
    We leave the details to the reader.
\end{proof}

\section{Concluding Remarks}

We consider three-pile games of Halve Nim in this paper.
We found various Nim-like results for the $\P$-positions
of $a^1 b^1 c^1$ when $1\le c\le 9$.
We observe that, for $a, b \ge 30$
and $0\le c \le 3$, the $\mathcal{P}$-positions 
of $a^1 b^1 c^1$ occur when $a\oplus b\in\{0,\, 1,\, 2,\, 3\}$.
Then, for $a,b\ge 30$ and $4\le c \le 7$, the $\P$-positions occur when $(a+2)\oplus (b+2)\in\{2,\, 3,\, 4,\, 5\}$.
Also, for $a,b\ge 64$ and $8\le c \le 9$, the $\P$-positions occur
when a Nim-like condition on $a$, $b$, and $c$ is satisfied.
However, this Nim-like structure begins to fall apart when $c=8$
because this is the first instance where the halving operation affects
the $\P$-positions of games in the periodic range 
where $a,b\ge 64$.
Furthermore, the Nim-like structure completely falls apart when $c=10$.

Additionally, for $0\le c\le 3$, the games $a^1 b^1 c^1$ do not have a periodic range.
Instead, after the threshold values $14$ for $c\in\{0,\, 1\}$
and $30$ for $c\in\{2,\, 3\}$,
the values of $a\oplus b$ alternate between either 0 and 2, or 1 and 3
at the break values $7\cdot 2^n$,
where $n$ is a positive integer.
The alternating values of $a\oplus b$ at these break values is
due to the halving operation.
The combined results from Theorems \ref{thm:TwoHeapPposition},
\ref{thm:ThreeHeapab1Pposition}, \ref{thm:ThreeHeapab2Pposition},
and \ref{thm:ThreeHeapab3Pposition} ensure that for every $a,b\ge 30$
such that $a\oplus b\in\{0,\, 1,\, 2,\, 3\}$,
there is a value $0\le c\le 3$ such that $a^1 b^1 c^1$ is a
$\P$-position.
This allows for the appearance of 
periodic behavior of the $\P$-positions of
$a^1 b^1 c^1$ when $a,b\ge 30$ and $c\ge 4$.
The threshold value of 30 for $4\le c\le 7$ roughly doubles to
64 for $c\in\{8,\, 9\}$.
This suggests that, to ensure
periodic behavior of the $\P$-positions of the game $a^1 b^1 c^1$,
the threshold value for $a$ and $b$ 
doubles as $c$ doubles.
Also, it appears that it will be difficult to determine the
periodic behavior of 
the $\P$-positions of $a^1 b^1 c^1$ as 
$a$, $b$, and $c$ get larger.

%
%

\end{document}